%% file: Associative_submanifolds_of_TCS_G2_manifolds.tex
\documentclass[leqno,11pt]{article}

\usepackage{tikz}
\usepackage{tikz-cd}
\usepackage{amsmath}

\input{preamble/preamble.tex}

\author{Gorapada Bera}
\title{Associative submanifolds in twisted connected sum $G_2$-manifolds}
\date{\vspace{-5ex}}
\begin{document}
\maketitle
\begin{abstract}
We introduce a method to construct closed rigid associative submanifolds in twisted connected sum $G_2$-manifolds. More precisely, we prove a gluing theorem of asymptotically cylindrical (ACyl) associative submanifolds in ACyl $G_2$-manifolds under a transverse intersection hypothesis. This is analogous to the gluing theorem for $G_2$-instantons introduced in \cite{SaEarp2013}. We rephrase the hypothesis in the special cases where the ACyl associative submanifolds are obtained from holomorphic curves or special Lagrangians in ACyl Calabi--Yau $3$-folds, in terms of algebraic-geometric conditions and  topological conditions, respectively. This yields many rigid associative submanifolds with new topological types $S^3$, $\R\P^3$ or $\R\P^3\#\R\P^3$. 
\end{abstract}
\section{Introduction} 
A $G_2$-manifold is a Riemannian $7$-manifold whose holonomy group is contained in the exceptional Lie group $G_2$ or, equivalently a smooth $7$-manifold with a torsion free $G_2$-structure \cite[Introduction]{Joyce1996}. The group $G_2$ is one of the two exceptional holonomy groups in Berger's classification \cite[Theorem 3]{Berger1955} of Riemannian manifolds. There has been considerable enthusiasm in understanding the geometry of $G_2$-manifolds in the last decade. This surge in interest can be attributed, in part, to the role that $G_2$-manifolds play in $M$-theory of theoretical high energy physics, which is analogous to the role that Calabi--Yau $3$-folds play in string theory. Simple examples of $G_2$-manifolds include product Riemannian manifolds $S^1 \times Z$, where $Z$ is a Calabi--Yau $3$-fold. However, the holonomy group of these examples is contained in the proper subgroup $\SU(3)\subset G_2$. The most effective method to date of constructing compact manifolds whose holonomy is exactly $G_2$ is the twisted connected sum (TCS) construction, which glues a suitable matching pair of asymptotically cylindrical (ACyl) $G_2$-manifolds. Kovalev \cite{Kov03} pioneered this construction which was later extended by \citet{KL11}. They produce $G_2$-manifolds from matching pairs of ACyl Calabi--Yau $3$-folds which are constructed from Fano $3$-folds or $K3$ surfaces with non-symplectic involutions. \citet{CHNP15} have produced hundreds of thousands of $G_2$-manifolds by extending this construction to a larger class of ACyl Calabi--Yau $3$-folds which are obtained from semi-Fano $3$-folds. 

To define an enumerative invariant of $G_2$-manifolds which is unchanged under the deformation of $G_2$-metrics, \citet{Joyce2016} and \citet{Doan2017d} have outlined proposals which are based on counting closed associative submanifolds. An associative submanifold in a $G_2$-manifold is a $3$-dimensional  submanifold calibrated by the $3$-form defining the $G_2$-structure \cites[IV.2.A]{Harvey1982}[Chapter 12]{Joyce2007}. These are volume minimizing in their homology class and hence are minimal submanifolds. They can be regarded as analogous to holomorphic curves and special Lagrangians in Calabi--Yau $3$-folds.  \citet[Section 3]{Donaldson1998} made a different proposal to define an invariant of $G_2$-manifolds by counting $G_2$-instantons. But the $G_2$-instantons can degenerate by bubbling off along associative submanifolds, playing a crucial role again here \cite{Donaldson2009}. The twisted connected sum (TCS) $G_2$-manifolds and associative submanifolds inside them can be a testing ground for these enumerative theories.

 Holomorphic curves in Calabi--Yau $3$-folds can be constructed directly using algebraic geometry, but the process becomes significantly more challenging when attempting to construct associative submanifolds. Nevertheless, the ACyl Calabi--Yau $3$-folds can be obtained from building blocks, that is, smooth projective $3$-folds $Z$ together with a projective morphism $\pi:Z\to \C\P^1$ such that $X:=\pi^{-1}(\infty)$ is a smooth anti-canonical $K3$ surface together with some additional data. A natural idea for constructing associative submanifolds in TCS $G_2$-manifolds is therefore to find closed rigid holomorphic curves $C$ in one of the building blocks $Z$ avoiding $X$, and then deform the product $S^1\times C$. Unfortunately, it is not easy to find such holomorphic curves $C$, at least not when the building blocks are obtained from Fano $3$-folds $W$, because the anti-canonical bundle $-K_W$ is ample. \citet{CHNP15} overcome this situation by considering building blocks obtained from semi-Fano $3$ folds. They produce some closed rigid associative submanifolds diffeomorphic to $S^1\times S^2$ by finding rigid holomorphic lines in the semi-Fano $3$-folds.

 In this article, we improve significantly our ability to construct associative submanifolds using a gluing technique. This technique now can be applied to holomorphic curves in the building blocks $Z$ which intersect $X$. More generally, in \autoref{section Gluing ACyl associative submanifolds in ACyl $G_2$-manifolds} we prove a gluing theorem: \autoref{thm gluing associative submanifolds},  for a pair of ACyl associative submanifolds in a matching pair of ACyl $G_2$-manifolds. This is analogous to the gluing theorem for $G_2$-instantons introduced in \cite[Theorem 2.3.4]{SaEarp2013}. Clearly, the pregluing construction yields approximate associative submanifolds. These can be deformed to genuine associative submanifolds if they are unobstructed in their deformation theory. We impose a set of conditions: \autoref{hyp gluing}, that guarantee the unobstructedness. 
 
 Constructing examples of associative submanifolds using \autoref{thm gluing associative submanifolds} is challenging, as verifying the conditions in \autoref{hyp gluing} is technically difficult. To address this, we focus on twisted connected sum (TCS) $G_2$-manifolds built from ACyl Calabi--Yau 3-folds and reformulate \autoref{hyp gluing} for pairs of ACyl associative submanifolds arising from either ACyl holomorphic curves or ACyl special Lagrangian 3-folds. These reformulations appear in \autoref{sec Associative submanifolds by gluing of ACyl holomorphic curves} as \autoref{thm gluing asso from ACyl holo}, and in \autoref{sec Associative submanifolds by gluing of ACyl special Lagrangian $3$-folds} as \autoref{thm gluing with ACyl SL}.
Thanks to the result of \citet{Haskins2012}, which shows that ACyl Calabi--Yau $3$-folds are essentially equivalent to building blocks, \autoref{thm gluing asso from ACyl holo} relies only on purely algebro-geometric conditions. This yields associative 3-spheres in many TCS $G_2$-manifolds constructed from Fano $3$-folds. Separately, \autoref{thm gluing with ACyl SL} relies only on topological conditions once ACyl special Lagrangian $3$-folds are given. Particularly in cases involving fixed-point sets of anti-holomorphic involutions, this helps us to produce new examples of rigid associative submanifolds diffeomorphic to $\R\P^3$ and $\R\P^3 \# \R\P^3$.

We note that many results in the literature construct objects satisfying geometric PDEs by gluing ACyl objects. Closest to our setting is work of \citet{Talbot2017}, who glues ACyl special Lagrangians to produce closed ones. In most such constructions, both the ACyl objects with a fixed cross-section and the cross-section itself are unobstructed in their respective deformation theories, making the gluing process relatively straightforward.
In contrast, our setting requires verifying \autoref{hyp gluing}, which involves a matching pair of ACyl associative submanifolds whose infinitesimal bounded deformations satisfy certain transversality conditions. Our analysis is similar to the work of  \citet{SaEarp2013} for $G_2$-instantons, although the PDEs differ, necessitating independent verification of all analytic estimates.
The main advantage of this article lies in \autoref{sec Associative submanifolds by gluing of ACyl holomorphic curves} and \autoref{sec Associative submanifolds by gluing of ACyl special Lagrangian $3$-folds}, where examples can be constructed easily. This contrasts with \cite{SaEarp2013}, where constructing examples is more difficult. A promising direction for future research is to construct $G_2$-instantons that degenerate via bubbling along the associative submanifolds developed here, following the work of \citet{Walpuski2013a}. Another direction is to establish uniqueness of the associative submanifold within the homology class of the one constructed in this work whenever the neck length is sufficiently large.

\paragraph{Acknowledgements.} I am grateful to my PhD supervisor Thomas Walpuski for suggesting the problem solved in this article and for his constant encouragement and advice. Additionally, I extend my thanks to Johannes Nordstr\"om, Jason Lotay, Dominik Gutwein, Viktor Majewski and anonymous referees for their feedbacks on previous versions. This material is based upon work supported by the \href{https://sites.duke.edu/scshgap/}{Simons Collaboration on Special Holonomy in Geometry, Analysis, and Physics} and the \href{https://nsf.gov/awardsearch/showAward?AWD_ID=1928930}{NSF under Grant No. DMS-1928930} during my Fall 2022 residency at the \href{https://www.msri.org}{Simons Laufer Mathematical Institute}. 

\paragraph{Convention.} \textit{Choose} a cut-off function $\chi\in C^\infty(\R, [0,1])$ with $\chi |_{(-\infty,0])}=0$ and $\chi |_{[1,\infty))}=1$. For a fixed $T\geq 0$ set  $\chi_T(t):=\chi(t-T)$ for $t\in \R$.

\section{The twisted connected sum (TCS) construction}\label{section TCS}
In this section we review the twisted connected sum construction of $G_2$-manifolds following \cite{CHNP15}.  
  
A  $3$-form $\phi$ on a $7$-dimensional manifold $Y$ is called \textbf{definite} if the bilinear form $G_\phi: S^2TY\to \Lambda^7(T^*Y)$ defined by 
$G_\phi(u,v):=\iota_u\phi\wedge\iota_v \phi\wedge \phi$
is definite. It  uniquely defines a Riemannian metric $g_\phi=\inp{\cdot}{\cdot}$ and a volume form $\vol_{g_\phi}$ on $Y$ satisfying the identity: $G_\phi=6g_\phi \otimes \vol_{g_\phi}$ and $\vol_{g_\phi}$ is the Riemannian volume form of $g_\phi$. Moreover it defines 
\begin{itemize}
\item a \textbf{cross product} $\times:\Lambda^2(TY)\to TY$, given by $\phi(u,v,w):=\inp{u\times v}{w},$	
\item an \textbf{associator} $[\cdot,\cdot,\cdot]:\Lambda^3(TY)\to TY$, given by $[u,v,w]:=(u\times v)\times w+\inp{v}{w}u-\inp{u}{w}v,$
\item a \textbf{$4$-form} $\psi:=*_{g_\phi} \phi\in \Omega^4(Y)$, or equivalently given by $\psi (u,v,w,z):=\inp{[u,v,w]}{z}$.\qedhere
\end{itemize} 
 \begin{definition}\label{def G2 manifold}
 A \textbf{$G_2$-manifold} is a $7$-dimensional manifold $Y$ equipped with a torsion-free $G_2$-structure, that is, equipped with a \textbf{definite} $3$-form $\phi\in \Omega^3(Y)$ such that $\nabla_{g_\phi}\phi=0$,
 or equivalently
 \begin{equation*}d\phi=0\ \ \text{and}\ \ d\psi=0.\qedhere \end{equation*} 
 \end{definition}

\begin{definition}\label{def ACyl G2}
	Let $(Z,\omega,\Omega)$ be a compact Calabi--Yau $3$-fold, where $\omega$ is a K\"ahler form and $\Omega$ is a holomorphic volume form satisfying
	$\tfrac {\omega^3}{3!}=-(\tfrac i 2)^3\Omega\wedge \bar \Omega$.
	 A $G_2$-manifold $(Y,\phi)$ is called an \textbf{asymptotically cylindrical (ACyl) $G_2$-manifold} with asymptotic cross section $(Z,\omega,\Omega)$ and rate $\nu<0$ if there exist 
	\begin{itemize}
	\item 	a compact submanifold $K_Y$ with boundary and a diffeomorphism $$\Upsilon:\R^+\times Z\to Y\setminus K_Y,$$
	\item  a $2$-form $\varrho$ on $\R^+\times Z$ such that $\Upsilon^*\phi=dt\wedge\omega+\operatorname{Re}\Omega+d\varrho$ with
$$ \abs{\nabla^k\varrho}=O(e^{\nu t})\ \text{as} \ t\to \infty, \forall k\in \N\cup\{0\}.$$ 
	\end{itemize}
 Here $t$ denotes the coordinate on $\R^+$, $\abs{\cdot}$ and Levi-Civita connection $\nabla$ are induced by the product metric on $\R^+\times Z$.
\end{definition}

\begin{remark}\label{rmk product ACyl G_2}
	Let $(V,\omega,\Omega)$ be an ACyl Calabi--Yau $3$-fold with asymptotic cross section $S^1\times X$, where $(X,\omega_1,\omega_2,\omega_3)$ is a compact hyperk\"ahler $4$-manifold  \cite[Definition 3.3]{CHNP15}. Then $$(Y:=S^1\times V,\phi:=d\theta\wedge\omega+\operatorname{Re}\Omega)$$ is an ACyl $G_2$-manifold with asymptotic cross section 
	\[(S^1\times S^1\times X,ds \wedge d\theta+\omega_3, (d\theta-ids)\wedge(\omega_1+i\omega_2)).\]
In the above, $s$ and $\theta$ denote the coordinates on the unit circles $S^1$.
\end{remark}
The following discussion summarizes the relationship between simply connected irreducible (i.e. not isometric to Riemannian product) ACyl Calabi--Yau $3$-folds and {building blocks}. 
	\begin{definition}
A \textbf{building block}\label{def building block} is a pair $(Z,X)$ in which $Z$ is a non-singular complex projective $3$-fold with primitive anti-canonical class $-K_Z$ in $H^2(Z)$, and $X\in\abs{-K_Z}$ is a smooth $K3$ surface divisor having trivial holomorphic normal bundle; equivalently there exists a projective morphism $\mathfrak f:Z\to \C\P^1$ with $\mathfrak f^{-1}(\infty)=X\in\abs{-K_Z}$ a smooth $K3$ surface.

A \textbf{framing} of a building block $(Z,X)$ is a hyperk\"ahler structure $\pmb{\omega}=(\omega_1,\omega_2,\omega_3)$ on $X$ such that $\omega_2+i\omega_3$ is of type $(2,0)$ and $[\omega_1]\in H^{1,1}(X)$ is the restriction of a Kähler class on $Z$. A \textbf{framed building block}\ is such a triple $(Z,X,\pmb{\omega})$. By Yau's proof of the Calabi conjecture,  each building block admits a framing.
\end{definition}

\begin{theorem}[{\citet[Theorem C,D]{Haskins2012}}]\label{thm ACyl and building block} Let $(Z,X,\pmb{\omega})$ be a framed building block. Then $V:=Z\setminus X$ is simply connected and there is a  irreducible ACyl Calabi--Yau structure $(\omega,\Omega)$ on $V$ with asymptotic cross section $S^1\times (X,\pmb{\omega})$. Conversely, let $(V,\omega,\Omega)$ be a simply connected irreducible ACyl Calabi--Yau $3$-fold with asymptotic cross section $S^1 \times (X,\pmb{\omega})$. Suppose $X$ is simply connected. Then there is a complex projective $3$-fold $Z$ such that $X\in\abs{-K_Z}$ is a divisor, $V$ is biholomorphic to $Z\setminus X$ and $(Z,X,\pmb{\omega})$ is a framed building block.
\end{theorem}
The following summarizes two sources of building blocks.
\begin{definition}\label{def weak and semi Fano}A \textbf{Fano $3$-fold} is a smooth projective $3$-fold $W$ such that the anti-canonical line bundle $-K_W$ is ample.

	A \textbf{weak Fano $3$-fold} is a smooth projective $3$-fold $W$ such that the anti-canonical line bundle $-K_W$ is nef and big, that is, $-K_W\cdot C\geq 0$ for all compact algebraic curves $C$ in $Z$ and $-K_W^3>0$.
	
	 A weak Fano $3$-fold $W$ is called \textbf{semi-Fano} if the anticanonical morphism $$W\longrightarrow R(W,-K_W):=\bigoplus_{l\geq 0} H^0(W,-lK_W)$$
	is {semi-small}, that is, it does not contract any divisor to a point.
	\end{definition}

\begin{theorem}[{\citet[Proposition 4.25]{CHNP13}}]\label{thm weak Fano to ACyl building block}
Let $W$ be a weak Fano $3$-fold, and suppose that $\abs{X_0,X_\infty}\subset \abs{-K_W}$ is an anti-canonical pencil with smooth (reduced) base locus $B$, and that $X\in \abs{X_0,X_\infty}$ is a smooth divisor. Let $Z$ be the blow-up of $W$ along the base locus $B$. Denote the proper transform of $X$ by $X$ again. Then $(Z,X)$ is a building block and $V:=Z\setminus X$ admits an ACyl Calabi--Yau structure.
 \end{theorem}
\begin{remark}
There are precisely $105$ deformation families of Fano $3$-folds \cite[Chapter 12]{Iskovskih1999} and all but two have a choice of pencils as described in  \autoref{thm weak Fano to ACyl building block} \cite[Proposition 3.15]{CHNP15}. The deformation families of weak Fano $3$-folds are also finite but there exist at least hundreds of thousands. Again, all but a few have a choice of pencils as described in  \autoref{thm weak Fano to ACyl building block} \cite[Section 4, Theorem 4.13]{CHNP13}.
\end{remark}
\begin{definition}\label{def non-symplectic involution}
A holomorphic involution $\rho$ on a $K3$ surface $X$ is called \textbf{non-symplectic} if $\rho^*(\alpha)=-\alpha$ for all $\alpha\in H^{2,0}(X)$.
\end{definition}
\begin{theorem}[{\citet[Proposition 5.1]{KL11}}]\label{thm Kovalev-Lee}
Let $X$ be a $K3$ surface with a non-symplectic involution $\rho$. Suppose that the fixed point locus $C:=\operatorname{Fix}_\rho(X)$ of $\rho$ is nonempty. Denote by $W:=\frac{\C\P^1\times X}{\iota\times \rho}$, where $\iota:\C\P^1\to \C\P^1$ is defined by $\iota(z)=\frac 1z$.  Let $Z$ be the blow-up of $W$ along the singular set $\{\pm 1\}\times C$. Denote the proper transform of the equivalence class of $\{\infty\}\times X$ by $X$ again. Then $(Z,X)$ is a building block and $V:=Z\setminus X$ admits an ACyl Calabi--Yau structure.
 \end{theorem}

\begin{remark}
There are exactly $75$ deformation families of $K3$ surfaces with non-symplectic involutions  and all but one	satisfy the nonempty assumption of the fixed point locus in \autoref{thm Kovalev-Lee} \cite[Proposition 3.2]{KL11}.
\end{remark}

The following summarizes the twisted connected sum construction.

\begin{definition}A pair of ACyl $G_2$-manifolds $(Y_{\pm},\phi_{\pm})$ with asymptotic cross sections $(Z_\pm,\omega_\pm,\Omega_\pm)$ is said to be a \textbf{matching pair} if there exists a diffeomorphism $f:Z_+\to Z_-$ such that
\begin{equation*}
f^*\omega^-=-\omega^+,\ \ f^*\Re\Omega^-=\Re\Omega^+.
 \qedhere
 \end{equation*}
\end{definition}

 Let $(Y_{\pm},\phi_{\pm})$ be a matching pair of ACyl $G_2$-manifolds. Let $\Upsilon_\pm:\R^+\times Z_\pm\to Y_\pm\setminus K_{Y_\pm}$ be the diffeomorphisms in \autoref{def ACyl G2}.  For $T\geq 1$, the compact $7$-manifold $Y_T$ is defined by
 $$Y_T:=Y_{T,+}\cup_F Y_{T,-}$$
 where $Y_{T,\pm}:=K_{Y_\pm}\cup\Upsilon_\pm((0,T+1]\times Z_\pm)$ and $F:[T,T+1]\times Z_+\to [T,T+1]\times Z_-$ is given by
 $$F(t,z)=(2T-t+1,f(z)).$$
In summary, $Y_T$ is obtained by gluing $Y_{T,\pm}$ through the identification map $F$. 
 The $3$-form $\hat\phi_T$ on $Y_T$ defined by
 $$\hat\phi_T:=\phi_\pm-d\big((\Upsilon_\pm^{-1})^*\chi_{T-1}\varrho_\pm\big) \ \ \text{on} \ \ Y_{T,\pm},$$
 is a closed $G_2$-structure.  For all sufficiently large $T$ (as the error is small enough) the following theorem allows to deform it to a torsion free $G_2$-structure.

\begin{theorem}[{\citet[Theorem 5.34]{Kov03}}]\label{thm Kovalev tcs}Let $(Y_{\pm},\phi_{\pm})$ be a matching pair of ACyl $G_2$-manifolds. Then there exist constants $T_0>1$, $\delta>0$ and a unique torsion free $G_2$ structure $\phi_T$ on $Y_T$ with $[\phi_T]=[ \hat \phi_T]$ for each $T\geq T_0$ such that
	\begin{equation}\label{eq Kovalev thm}
		\norm{\phi_T-\hat \phi_T}_{C^{k,\gamma}}=O(e^{-\delta T}), \quad  \forall k\in \N\cup\{0\}, \gamma\in (0,1).
	\end{equation}
\end{theorem}
 \begin{definition} The $G_2$-manifold $(Y_T,\phi_T)$ in \autoref{thm Kovalev tcs} is called a  \textbf{twisted connected sum} $G_2$-manifold.
 	\end{definition}
	
The following summarizes sources of matching pairs of ACyl $G_2$-manifolds.
\begin{definition}Let $(X_\pm,\omega_{1}^\pm,\omega_{2}^\pm,\omega_{3}^\pm)$ be a pair of compact hyperkähler $4$-manifolds. A diffeomorphism $\mathfrak r:X_+\to X_-$ is said to be a \textbf{hyperkähler rotation} if 
\begin{equation*}\mathfrak r^*\omega_{1}^-=\omega_{2}^+,\ \ \mathfrak r^*\omega_{2}^-=\omega_{1}^+ \ \ \text{and}\ \ \mathfrak r^*\omega_{3}^-=-\omega_{3}^+.\qedhere
\end{equation*}
\end{definition}
\begin{definition}\label{def TCS G2}Let $(V_\pm,\omega_\pm,\Omega_\pm)$ be a pair of ACyl Calabi--Yau $3$-folds with compact hyperkähler asymptotic cross sections $(X_\pm,\omega_{1}^\pm,\omega_{2}^\pm,\omega_{3}^\pm)$ and $\mathfrak r:X_+\to X_-$ be a hyperkähler rotation. Then $Y_\pm:=S^1\times V_\pm$ is a matching pair of  ACyl $G_2$-manifolds matched by the diffeomorphism $$f:S^1\times S^1\times X_+\to S^1\times S^1\times X_-,$$
	 defined by $f(\theta,s,x)=(s,\theta,\mathfrak r(x))$. Therefore by \autoref{thm Kovalev tcs}, we have a family of compact twisted connected sum $G_2$-manifolds $(Y_T,\phi_T)$ for all sufficiently large $T$. 
	\end{definition} 
\begin{remark}
	Given a pair of building blocks $(Z_\pm,X_\pm)$ there are no systematic ways to find hyperkähler rotations $\mathfrak r:X_+\to X_-$. However, \citet[Proposition 6.18, Proposition 6.2, Remark 6.19]{CHNP15} discuss the existence of hyperkähler rotations when the building blocks arise from semi-Fano $3$-folds.
\end{remark}

\section{Asymptotically cylindrical (ACyl) associative submanifolds}\label{subsec Asymptotically cylindrical (ACyl) associative submanifolds}
\citet{Harvey1982} considered a special class of $3$-dimensional calibrated submanifolds of $G_2$-manifolds, called associative submanifolds. This section provides definition and examples of  ACyl associative submanifolds in ACyl $G_2$-manifolds.
\begin{definition}\label{def closed associative}Let $(Y,\phi)$ be a $G_2$-manifold. A $3$-dimensional oriented submanifold $P$ of $Y$ is called an \textbf{associative submanifold} if it is calibrated by the $3$-form $\phi$, that is, $\phi|_{_P}$ is the volume form $\vol_{P,g_\phi}$ on $P$, or equivalently $\phi|_{_P}$ is the orientation and $[u,v,w]=0$, for all $x\in P$ and $u,v,w\in T_xP$. 
\end{definition}

\begin{definition}\label{def ACyl asso} Let $(Y,\phi)$ be an ACyl $G_2$-manifold with asymptotic cross section $(Z,\omega,\Omega)$ and  rate $\nu<0$, equipped with the diffeomorphism $\Upsilon:\R^+\times Z\to Y\setminus K_Y $ as described in \autoref{def ACyl G2}. Let $C=\R\times \Sigma$ be a cylinder in $\R\times Z$. Let $\Sigma=\amalg_{i=1}^m \Sigma_i$ be the decomposition of $\Sigma$ into connected components, and subsequently $C=\amalg_{i=1}^mC_i$, where $C_i=\R\times \Sigma_i$. Let $\Upsilon_C:V_C\to U_C\subset \R\times Z$ be a translation invariant tubular neighbourhood map of $C$. 

 A smooth three dimensional oriented embedded submanifold $P$ of $Y$ is said to be an \textbf{asymptotically cylindrical (ACyl) submanifold} with asymptotic cross section $\Sigma$ and rate $\mu=(\mu_1,\mu_2,...,\mu_m)$ with $\nu\leq \mu_i<0$ for all $i=1,2...,m$ if there exist 
 \begin{itemize}
 	\item a compact submanifold with boundary $K_P$ of $P$,
 	\item  a constant $T_0>0$, and a smooth embedding 
	$\Psi_P:(T_0,\infty)\times \Sigma\to U_C\subset \R^+\times Z$
	such that $\Upsilon\circ \Psi_P:(T_0,\infty)\times \Sigma\to Y$ is a diffeomorphism onto $P\setminus K_P$ and 
		$\Psi_P=\Upsilon_C\circ\alpha$ over $(T_0,\infty)\times \Sigma$ for some smooth section $\alpha$ of the normal bundle $NC$ of $C$ which lies in $V_C$ and
		 \begin{equation}\label{eq ACyl asso}
		 	\abs{(\nabla^\perp_{C_i})^k\alpha}=O(e^{\mu_i t}) \  \text{as} \  t\to \infty, i=1,2...,m,\ \forall k\in \N\cup\{0\}.
		 			 	 \end{equation}
		 	  \end{itemize}
 Here $\nabla^\perp_{C}$ is the normal connection on $NC$ induced from the Levi-Civita connection on $\R^+\times Z$ and $\abs{\cdot}$ is respect to the normal metric on $NC$ and cylindrical metric on $C$.
		$P$ is said to be an \textbf{ACyl associative submanifold} if it is associative and ACyl submanifold as above.	
\end{definition}

\begin{example}\label{eg ACyl asso from ACyl holo}
	Let $(V,\omega,\Omega)$ be an ACyl Calabi--Yau $3$-fold with asymptotic cross section $X$ and let $Y:=S^1\times V$ be the ACyl $G_2$-manifold as described in  \autoref{rmk product ACyl G_2}. 
		\begin{enumerate}[label=(\roman*), leftmargin=*]	
		\item Let $\sC^*$ be an ACyl embedded holomorphic curve in $V$ with asymptotic cross section $\amalg_{j=1}^mS^1\times \{x_j\}$ in $S^1\times X$. Then $P:=S^1\times \sC^*$ is an ACyl associative submanifold in $Y$, whose asymptotic cross section is $\Sigma:=\amalg_{j=1}^m T^2\times \{x_j\}$.
	\item Let $(Z,X)$ be a building block as described in \autoref{def building block} and let $V:=Z\setminus X$ be the corresponding ACyl Calabi--Yau $3$-fold; see \autoref{thm ACyl and building block}. Let $\sC$ be a closed embedded holomorphic curve in $Z$ intersecting $X$ transversely at $\bar x:=\{x_1,\dots,x_m\}$. Then $\sC^*:=\sC\setminus X$ is an ACyl embedded holomorphic curve in $V$, whose asymptotic cross section is $\amalg_{j=1}^m S^1\times \{x_j\}$. This is proved in \autoref{Lem ACyl asso from holo in building blocks}.
	\item Let $W$ be a weak-Fano $3$-fold and let $B$ be the base locus of an anti-canonical pencil. Let $\pi:Z\to W$ be the blow up of $W$ along $B$ as given in \autoref{thm weak Fano to ACyl building block}. 
		\begin{enumerate}[a)]
		\item Let $\sC$ be an embedded holomorphic curve in $W$ avoiding $B$ and satisfying ${-K_W}\cdot \sC>0$. Then for general $X\in \abs{-K_W}$ in the anti-canonical pencil, the proper transform of $\sC$ in $Z$ is an example of a closed embedded holomorphic curve in part (ii), where the building block is $(Z,X)$. 
		\item For each $b\in B$, $\ell_b:=\pi^{-1}(b)$ is an embedded rational curve in $Z$ which intersects $X$ transversely at one point and $$N\ell_b\cong \mathcal O_{\P^1}\oplus \mathcal O_{\P^1}(-1). $$
		This is another example of a closed embedded holomorphic curve in part (ii).  
		\end{enumerate} 
	\item Let $X$ be a $K3$ surface with a non-symplectic involution $\rho$ and let $\pi:Z\to W$ be the blow-up given in \autoref{thm Kovalev-Lee} so that $(Z,X)$ is a building block. Denote by $\operatorname{Fix}_\rho(X)$ the fixed point locus of $\rho$. 
	\begin{enumerate}[a)]
	\item Let $x\notin \operatorname{Fix}_\rho(X)$. Then the proper transform of $\P^1\cong [\P^1\times \{x\}]\subset W$ in $Z$, denoted by $\ell_x$, is a closed embedded holomorphic curve in part (ii) intersecting $X$ at two points, namely $x$ and $\rho(x)$. In this case, 
	$$N\ell_x\cong \mathcal O_{\P^1}\oplus \mathcal O_{\P^1}.$$ 
	\item Let $y\in \operatorname{Fix}_\rho(X)$. Then the proper transform of ${\P^1}/{\Z_2}\ \cong[\P^1\times \{y\}]\subset W$ in $Z$, denoted by the line $\ell_y$, is another example of a closed embedded holomorphic curve intersecting $X$ at the single point $y$. In this case,
	 \begin{equation*}
 	N \ell_y\cong
	\mathcal O_{\P^1}\oplus \mathcal O_{\P^1}(-1).	  
\qedhere
\end{equation*}
\end{enumerate} 
\end{enumerate}
   \end{example}

\begin{remark}
Let $(Y,\phi)$ be an ACyl $G_2$-manifold with asymptotic cross section $(Z,\omega,\Omega)$. A nontrivial $G_2$-involution $\sigma:Y\to Y$ (i.e.
$\sigma^2=\id,\sigma^*\phi=\phi$) is always an ACyl $G_2$-involution \cite[Proposition 2.3.7]{Nordstrom2008th}, that is, there exist constants $T>0$, $\epsilon>0$, a non-trivial $SU(3)$-involution $\tau:Z\to Z$ (i.e. $\tau^2=\id, \tau^*\omega=\omega$ and $\tau^*\Omega=\Omega$) and a vector field $N$ on $\R^+\times Z$ such that over $(T,\infty)\times Z$,  
$\sigma\circ \Upsilon=\Upsilon \circ\exp(N\circ (\id\times\tau ))$
with $${\abs{\nabla^kN}}=O(e^{-\epsilon t}) \  \text{as} \  t\to \infty, \forall k\in \N\cup\{0\}.$$ 
 Here $t$ denotes the coordinate on $\R^+$, $\abs{\cdot}$ and Levi-Civita connection $\nabla$ are induced by the product metric on $\R^+\times Z$. Let $P$ be a connected non-compact component of $\operatorname{Fix}_\sigma(Y)$. Then it  is an ACyl associative \cite[Proposition 12.3.7]{Joyce2007} with cross section $\Sigma\subset \operatorname{Fix}_\tau(Z)$.
\end{remark}	
\begin{example}\label{eg ACyl asso from ACyl SL}
	Let $(V,\omega,\Omega)$ be an ACyl Calabi--Yau $3$-fold with compact hyperkähler asymptotic cross section $(X,\omega_1,\omega_2,\omega_3)$ and $Y:=S^1\times V$ be the ACyl $G_2$-manifold as described in \autoref{rmk product ACyl G_2}. 
	\begin{enumerate}[label=(\roman*), leftmargin=*]	
	\item Let $L$ be an ACyl embedded special Lagrangian $3$-fold in $V$ with asymptotic cross section $\{e^{is}\}\times\Sigma$, where $\Sigma$ is an embedded $I_3$-holomorphic curve in $X$. Then for each $\theta\in [0,2\pi)$, $L_\theta:=\{e^{i\theta}\}\times L$ is an ACyl associative in $Y=S^1\times V$ with cross section $\{e^{i\theta}\}\times\{e^{is}\}\times \Sigma$. By abusing notation we will denote them by $L$ and  $\Sigma$ respectively. 
	\item Let $\sigma_V:V\to V$ be an anti-holomorphic involutive isometry (i.e. $\sigma_V^2=\id, \sigma_V^*\omega=-\omega,\sigma_V^*\Omega=\bar\Omega$).  Let $L$ be a connected non-compact component of the fixed point locus $\operatorname{Fix}_{\sigma_V}(Z)$. Then it is an ACyl special Lagrangian $3$-fold \cite[Proposition 3.11]{Talbot2017}. Thus we can apply part (i) to it.  In this case, the ACyl associative $L_\theta\sqcup L_{\theta+\pi}$ is also the fixed point locus of the $G_2$-involution  $\sigma_\theta:S^1\times V\to S^1\times V$ defined by $\sigma_\theta(e^{it}, z)=(e^{i(2\theta-t)},\sigma_V z)$.
	\item Let $(Z,J,X,\pmb \omega)$ be a framed building block as described in \autoref{def building block} and let $\sigma_Z:Z\to Z$ be an anti-holomorphic involution (i.e. $\sigma_Z^2=\id, \sigma_Z^*J=-J$) that restricts to an anti-holomorphic involutive isometry on $(X,\pmb \omega)$. Then $V:=Z\setminus X$ admits an ACyl Calabi--Yau structure $(\omega,\Omega)$ such that $\sigma_Z$ is an anti-holomorphic involutive isometry on $V$ \cite[Proposition 5.2]{Kovalev2013}. If a building block $(Z,X)$ admits an anti-holomorphic involution $\sigma_Z$ which takes $X$ to $X$, then there is a framing $\pmb \omega$ such that it is an anti-holomorphic involutive isometry on $(X,\pmb \omega)$. Thus we can apply part (ii) to it.

	\item Building blocks obtained from weak Fano $3$-folds (see \autoref{thm weak Fano to ACyl building block}) that possess anti-holomorphic involution preserving both the anti-canonical divisor and base locus, admit a lifting of the anti-holomorphic involution \cite[pg. 19]{Kovalev2013}, thereby satisfying the conditions of part (iii).
	\item  Building blocks obtained from $K3$ surfaces (see \autoref{thm Kovalev-Lee}) that possess commuting non-symplectic involution and anti-holomorphic involution, admit a lifting of the anti-holomorphic involution. This type of $K3$ surfaces has been studied in \cite{Nikulin2005,Nikulin2007}.\qedhere
	 \end{enumerate}
   \end{example}

\section{Linear analysis on ACyl associative submanifolds}\label{subsec Linear analysis on ACyl associative submanifolds}
In this section, we review the linear analysis on ACyl associative submanifolds that is essential for the gluing theorem presented in the next section. For further details, the reader is referred to \cite[Section 4]{Bera2025} or \cite[Section 3.1]{SaEarp2013}, as well as the references therein.
  
  Let $(Y,\phi)$ be an ACyl $G_2$-manifold with asymptotic cross section $(Z,J,\omega,\Omega)$ and rate $\nu<0$ as described in  \autoref{def ACyl G2}. Let $P$ be an ACyl associative submanifold asymptotic to a cylinder $C=\R\times \Sigma$ with rate $\mu\geq \nu$ as described in  \autoref{def ACyl asso}. There is a Dirac operator that controls the deformation theory of associative submanifolds, called the \textbf{Fueter operator} $\mathbf D_{P}:C_c^\infty(NP)\to C_c^\infty(NP)$, defined by
 \begin{equation}\label{eq Fueter operator}
 \mathbf D_{P}:=\sum_{i=1}^3e_i\times \nabla^{\perp}_{P,e_i}
 \end{equation}
 where $NP$ is the normal bundle of $P$ and $\nabla^{\perp}_P$ is the normal connection and  $\{e_1,e_2,e_3\}$ is any local oriented orthonormal frame for $TP$ with respect to the metric $g_\phi$. A straightforward computation shows that the above definition is independent of the choice of local frame.  
 
 The Fueter operator $\bD_C$ on the associative cylinder $C$ is defined in the same way as in \autoref{eq Fueter operator}, with $P$ replaced by $C$. There is also a similar Dirac operator $\bD_\Sigma: C^\infty(N\Sigma)\to C^\infty(N\Sigma)$ on the $J$-holomorphic curve $\Sigma$, which also serves as its deformation operator, defined by 
\begin{equation}\label{eq: Fueter on Riemann surface} 
\mathbf D_\Sigma:= \sum_{i=1}^2f_i\times \nabla_{f_i}^\perp
\end{equation}
where $\{f_i\}$ is any local orthonormal oriented frame on $\Sigma$. Again, this definition is independent of the choice of local frame. Observe that $\ker \bD_\Sigma$ is the space of holomorphic normal vector fields of $\Sigma$. The operator $\bD_C$ is translation invariant (i.e. independent of the coordinate $t$ on $\R$) and takes the form 
  \[\bD_C=J\partial_t+\bD_\Sigma.\]
This is established in \cite[Proposition 4.3(i)]{Bera2025}. Furthermore, \cite[Proposition 5.21]{Bera2025} demonstrates that $\bD_{P}$ is an asymptotically translation invariant uniformly elliptic operator asymptotic to $\mathbf D_C$ at rate $\mu$, as defined in \cite[Definition 4.13]{Bera2025}.

The Fredholm theory for asymptotically translation invariant uniformly elliptic operators is well-established and thoroughly treated in \cite[Section 4.3]{Bera2025}. In what follows, we present only the essential aspects, following the approach of \cite[Section 3.1]{SaEarp2013}. Let $\Sigma=\amalg_{i=1}^m \Sigma_i$ be the decomposition of $\Sigma$ into connected components. Then $C=\amalg_{i=1}^mC_i$, where $C_i=\R\times \Sigma_i$. We would like to define the weighted Hölder spaces with rate $\lambda=(\lambda_1,\lambda_2,\dots,\lambda_m)\in \R^m$. For this we choose a \textbf{weight function}  
	$w_{P,\lambda}:P\to(0,\infty)$, which is a smooth function on $P$ satisfying for all $x=\Upsilon\circ\Psi_P (t,\sigma)$ with $(t,\sigma)$ in $(T_0,\infty)\times \Sigma_i$,
		$$w_{P,\lambda}(x)= e^{-\lambda_i t}.$$
Let $ k\in \N\cup\{0\}$ and  $\gamma\in (0,1)$. For a continuous section $u$ of $NP$ we define the \textbf{weighted Hölder semi-norm} by 
	$$ [u]_{C_{P,\lambda}^{0,\gamma}}:=[w_{P,\lambda}u]_{C^{0,\gamma}({NP})}.$$ 
For a continuous section $u$ of $NP$ with $k$ continuous derivatives we define the \textbf{weighted $C^k$ norm} and the \textbf{weighted Hölder norm}, respectively, by
	$$\norm{u}_{C^{k}_{P,\lambda}}:=\sum_{j=0}^{k}\norm{(\nabla_{{P}}^\perp)^ju}_{L^\infty_{P,\lambda}},\ \ \norm{u}_{C^{k,\gamma}_{P,\lambda}}:=\sum_{j=0}^{k}\norm{(\nabla_{P}^\perp)^ju}_{L^\infty_{P,\lambda}}+[(\nabla_{{P}}^\perp)^ku]_{C^{0,\gamma}_{P,\lambda}}.$$ 
We define the \textbf{weighted Hölder space} $C^{k,\gamma}_{P,\lambda}$ to be the space of continuous sections of $NP$ with $k$ continuous derivatives and finite {weighted Hölder norm} $\norm{\cdot}_{C^{k,\gamma}_{P,\lambda}}$. Furthermore, we define the \textbf{weighted $C^\infty$-space $C^{\infty}_{P,\lambda}$} by
	\begin{equation*}C^{\infty}_{P,\lambda}:=\bigcap_{k=0}^\infty C^{k}_{P,\lambda}.\qedhere
	\end{equation*}

 Similarly, we define the weighted function spaces $C^{k,\gamma}_{C,\lambda}$, along with all other relevant spaces over $C$, by substituting $P$ with $C$, $NP$ with $NC$, and replacing the weight function $w_{P,\lambda}$ with $w_{C,\lambda}: C = \amalg_i C_i \to \R$, where $w_{C,\lambda}(t, \sigma_i) = e^{-\lambda_i t}$ for $\sigma_i \in \Sigma_i$. 
 
 The Fueter operators $\bD_{P}$ and $\bD_C$ admit natural extensions to the weighted function spaces, which we denote by
\begin{equation} \label{def Fueter on ACyl asso}
\bD_{P,\lambda}^{k,\gamma}:C^{k+1,\gamma}_{P,\lambda}\to C^{k,\gamma}_{P,\lambda}, \quad \bD_{C,\lambda}^{k,\gamma}:C^{k+1,\gamma}_{C,\lambda}\to C^{k,\gamma}_{C,\lambda}.\qedhere
\end{equation}
The \textbf{wall of critical rates} $\mathcal D_C$ is defined by
	$$\mathcal D_C:=\{(\lambda_1,\lambda_2,..,\lambda_m)\in \R^m: \lambda_i\in \spec(J\bD_{\Sigma_i}) \ \text{for some}\ i\}.$$ 
It is also noted in \cite[Proposition 4.3]{Bera2025} that $\spec(J\bD_{\Sigma})$ is symmetric with respect to zero and $\bD_\Sigma$ is $J$-anti-linear.	 

The following lemma is explained in \cite[Section 3]{Donaldson2002} and also discussed in \cite[Lemma 4.17, Proposition 4.18]{Bera2025}.
\begin{lemma}\label{lem fourier series Donaldson Acyl} The following hold.
	\begin{enumerate}[label=(\roman*), leftmargin=*]	
\item $\bD_{C,\lambda}^{k,\gamma}$ is invertible  if and only if $\lambda\in \R^m\setminus \mathcal D_C$. 
 Moreover, any element $u\in \ker \mathbf D_C$ has a unique decomposition
$$u=\sum_{\lambda\in \mathcal D_C}e^{\lambda t} u_{\Sigma,\lambda}$$
where $u_{\Sigma,\lambda}$ is a $\lambda$-eigensection of $J\mathbf D_\Sigma$.
\item  $\bD_{P,\lambda}^{k,\gamma}$ is Fredholm for all $\lambda\in \R^m\setminus\mathcal D_C$. Moreover, for all $\lambda\in \R^m$, $\ker \bD_{P,\lambda}^{k,\gamma}$ is finite dimensional, independent of $k$ and $\gamma$ and is consisting of smooth elements. If $\lambda \notin \mathcal D_C$ then
        \[\ker \bD_{P, \lambda}\cong \coker \bD_{P, -\lambda}.\]
\end{enumerate}
 \end{lemma}
As a consequence, we obtain the following proposition, which also appears in \cite[Proposition 3.5]{SaEarp2013} and in a more general form in \cite[Lemma 4.19]{Bera2025}. We include the proof here because it introduces the definition of the asymptotic limit map \autoref{eq asmp limit map}, which plays a crucial role in the gluing hypothesis discussed in the next section.
\begin{prop}\label{prop main Fredholm acyl}Set $\lambda_0:=\min\{\abs{\lambda}:0\neq \lambda\in \spec (J\bD_\Sigma), \lambda\geq \mu_i, i=1, \dots,m\} $. Then for all $s\in [0, \lambda_0)^m$, we have $\ker \bD_{P,s}=\ker \bD_{P,0} $. Moreover, there exists a unique linear map, referred to as the \textbf{asymptotic limit map},
\begin{equation}\label{eq asmp limit map}
\iota_\infty:\ker \bD_{P,0} \to \ker \bD_\Sigma
\end{equation} satisfying: for any $u\in \ker \bD_{P,0}$ and $s\in [0, \lambda_0)^m$,
	$u- \chi_{T_0} \iota_\infty u\in  C^{\infty}_{P, -s}$ (under the canonical normal bundle identifications);
in particular, \[\ker \iota_\infty=\ker \bD_{P, -s} \cong \coker \bD_{P, s} \qandq  \dim \im \iota_\infty=\frac{1}2 \dim \ker \bD_\Sigma.\]
\end{prop}
\begin{proof}
	Let $u$ be an element of $\ker \bD_{P,s}$, $s\in [0, \lambda_0)^m$. Set $\widetilde u:=\chi_{T_0}u\in C^{\infty}(NC)$. Since $\mu_i\leq -s_i$ for all $i=1, \dots,m$, $\mathbf D_C \widetilde u\in C_{C,-s}^{\infty}$ and therefore by \autoref{lem fourier series Donaldson Acyl} there exists a unique $v \in C_{C,-s}^{\infty}$ such that $\mathbf D_C (\widetilde u-v)=0$. Define $$ \iota_\infty(u):=(\widetilde u-v)_{\Sigma,0}\in \ker \bD_\Sigma$$ following the decomposition in \autoref{lem fourier series Donaldson Acyl}. Since $\widetilde u-v-(\widetilde u-v)_{\Sigma,0}\in C_{C,-s}^\infty $, we obtain $u- \chi_{T_0} \iota_\infty u\in  C^{\infty}_{P, -s}$. The last isomorphism: $\ker \bD_{P, -s} \cong \coker \bD_{P, s}$, follows from the fact that $\bD_P$ is formally self-adjoint \cite[Proposition 4.22 (i)]{Bera2025}. As a consequence, $\dim \im \iota_\infty=\ind \bD_{P,s}=-\ind \bD_{P,-s}$, with any $s \in [0, \lambda_0)^m$, which is same as ${1/2\dim \ker \bD_\Sigma}$ by the index jump formula \cite[Lemma 4.19 (iii)]{Bera2025}. 
	\end{proof}


\section{Gluing of ACyl associative submanifolds}\label{section Gluing ACyl associative submanifolds in ACyl $G_2$-manifolds}
Let $(Y_\pm, \phi_\pm)$ be a matching pair of asymptotically cylindrical (ACyl) $G_2$-manifolds, with a matching map $f: Z_+ \to Z_-$. Consider the family $\{(Y_T, \phi_T): T \geq T_0\}$ of twisted connected sum $G_2$-manifolds as described in \autoref{thm Kovalev tcs}. Let $\nu_\pm<0$ denote the asymptotic rates of $Y_\pm$. Let $P_\pm$ be a pair of ACyl associative submanifolds in $Y_\pm$, with asymptotic cross-sections $\Sigma_\pm$ and rates $\mu_\pm \geq \nu_\pm$, as described in \autoref{def ACyl asso}. The cross-sections $\Sigma_\pm$ correspond to holomorphic curves in the Calabi--Yau $3$-folds $Z_\pm$.
We can construct an approximate associative submanifold $P_T$ in the twisted connected sum $Y_T$ by ``pregluing'' $P_\pm$, and it can be perturbed to genuine associative submanifold if the deformation theory is unobstructed. To ensure this, we impose a following condition on $P_\pm$ that guarantees unobstructedness. While rigidity of $P_\pm$ would certainly suffice, that does not hold in our main application involving twisted connected sum $G_2$-manifolds. Instead, we impose the following hypothesis.

 \begin{hypothesis}\label{hyp gluing}
The ACyl associative submanifolds $P_\pm$ satisfy the following conditions:
\begin{itemize}
    \item The asymptotic cross-sections $\Sigma_\pm$ are matched via the identification map $f$, i.e., $f(\Sigma_+) = \Sigma_-$,
    
    \item  There are no infinitesimal deformations of $P_\pm$ fixing the asymptotic cross sections $\Sigma_\pm$. In other words, the asymptotic limit maps $\iota^\pm_\infty: \ker \mathbf{D}_{P_\pm,0} \to \ker \mathbf{D}_{\Sigma_\pm}$ are injective. Here, $\mathbf{D}_{\Sigma_\pm}$ are the deformation operators for $\Sigma_\pm$ as defined in \autoref{eq: Fueter on Riemann surface}, $\mathbf{D}_{P_\pm,0}$ are as defined in \autoref{def Fueter on ACyl asso}, and $\iota^\pm_\infty$ are introduced in \autoref{eq asmp limit map},
    
    \item  The images of $f_*\iota^+_\infty $ and $\iota^-_\infty$ intersect trivially, that is,
$\operatorname{im}(f_*\iota^+_\infty) \cap \operatorname{im}{\iota^-_\infty}=\{0\}.$
\end{itemize}
\end{hypothesis}

\begin{remark}The first condition in \autoref{hyp gluing} simply ensures that the asymptotic cross-sections of $P_\pm$ are matched under the map $f$. Whenever the asymptotic cross-sections $\Sigma_\pm$ are Morse–Bott, i.e., every infinitesimal deformation in $\ker \bD_{\Sigma_\pm}$ integrates to a holomorphic curve (see \cite[Definition 3.11]{Bera2025}), the final two conditions in the hypothesis can be expressed geometrically.
 The second condition is equivalent to the unobstructedness of $P_\pm$ in the deformation theory where the cross-sections are allowed to vary, implying in particular that the moduli spaces of ACyl associative submanifolds near $P_\pm$ are smooth (see \cite[Theorem 5.27]{Bera2025}). The third condition guarantees that the images of these moduli spaces intersect transversely at $\Sigma_\pm$ within the moduli space of holomorphic curves in $Z_\pm$. The hypothesis can be interpreted as a transverse Lagrangian intersection condition, as explained in \cite[Remark 1.1]{Bera2025}, where further details about these moduli spaces can also be found.

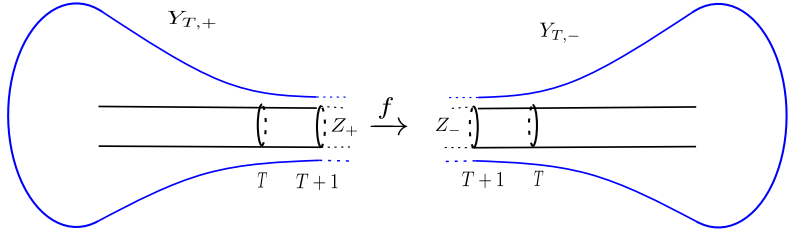
While the conditions stated in \autoref{hyp gluing} are sufficient for carrying out the gluing construction in this article, they are not strictly necessary. It is conceivable that the hypothesis could be weakened to require only a transversality condition for a family of matching pairs of ACyl associative submanifolds in a corresponding family of matching pair of ACyl $G_2$-manifolds. However, such generalizations will not be explored in this article. 
\end{remark}

By imposing the \autoref{hyp gluing} we prove the following gluing theorem.
\begin{theorem}\label{thm gluing associative submanifolds}Let $(Y_\pm,\phi_\pm)$ be a matching pair of ACyl $G_2$-manifolds with a matching map $f:Z_+ \to Z_-$ and let $\{(Y_T,\phi_T):T\geq T_0\}$ be the family of twisted connected sum $G_2$-manifolds. Let $P_\pm$ be a pair of ACyl associative submanifolds in $Y_\pm$ with asymptotic cross sections $\Sigma_\pm$ satisfying \autoref{hyp gluing}. Then there exist a constant $T_1\geq T_0$ and a family of smooth rigid associative submanifolds $\widetilde P_T$ in $(Y_T,\phi_T)$ for all $T\geq T_1$, which are all diffeomorphic to the twisted connected sum $P_+ \#_f P_-$, the $3$-manifold obtained by gluing $P_+$ and $P_-$ along their asymptotic cross-sections $\Sigma_\pm$ via the identification map $f$ (see \autoref{eqn tcsPT} for more explicit construction of $P_+ \#_f P_-$). 
\end{theorem}
   
The following two subsections prove the above theorem.

\subsection{Pregluing construction and estimates }\label{subsec Pregluing construction and estimates }  
Before constructing the approximate associative submanifolds via a pregluing construction, we first make a preparatory definition.

\begin{definition}\label{def P_C}
Let $P$ be an ACyl associative submanifold in an ACyl $G_2$-manifold $Y$, asymptotic to a cylinder $C = \R \times \Sigma$, and represented by a section $\alpha$ over the end $P \setminus K_P$, as described in \autoref{def ACyl asso}. Recall the diffeomorphism $\Upsilon: \R^+ \times Z \to Y \setminus K_Y$ from \autoref{def ACyl G2}, and the translation-invariant tubular neighbourhood map $\Upsilon_C: V_C \to U_C \subset \R \times Z$ from \autoref{def ACyl asso}.

We define an {end-cylindrical} (ECyl) submanifold $P_C$, diffeomorphic to $P$ but cylindrical at the ends, by
\[
P_C := K_P \cup (\Upsilon \circ \Upsilon_C)((1 - \chi_{T_0}) \alpha).
\]
Set 
\[
K_{P_C} := P_C \setminus \Upsilon((T_0 + 1, \infty) \times Z).
\]

A tubular neighbourhood map
\[
\Upsilon_{P_C}: V_{P_C} \to U_{P_C}
\]
of $P_C$ is called {end-cylindrical} (ECyl) if $V_{P_C}$ and $\Upsilon_{P_C}$ agree with $\Upsilon_*(V_C)$ and $\Upsilon \circ \Upsilon_C \circ \Upsilon_*^{-1}$, respectively, over the region $\Upsilon((T_0 + 1, \infty) \times \Sigma)$.

Given a choice of an ECyl submanifold $P_C$ and an ECyl tubular neighbourhood map $\Upsilon_{P_C}$, there exists a section $\beta$ of $N P_C$ with image in $V_{P_C}$ such that $\beta$ vanishes on $K_P$ and satisfies
\[
\Upsilon_* \circ \alpha = \beta \circ \Upsilon \quad \text{on} \quad \Upsilon((T_0 + 1, \infty) \times \Sigma),
\]
and such that $\Upsilon_{P_C}(\Gamma_\beta) = P$. 

There is also a canonical bundle isomorphism \cite[Definition~2.27]{Bera2023}, denoted by
\begin{equation}\label{def identification of normal bundle Acyl}
\Theta_{P}^C : N P_C \to N P. \qedhere
\end{equation}
\end{definition}

Let $P_\pm$ be a pair of ACyl associative submanifolds in $Y_\pm$ with asymptotic cross sections $\Sigma_\pm$ satisfying \autoref{hyp gluing}. For all $T \geq T_0$, we construct approximate associative submanifolds $P_T \subset Y_T$ as follows. We continue to use the notation $\Upsilon_\pm: \R^+ \times Z_\pm \to Y_\pm \setminus K_{Y_\pm}$ for the diffeomorphisms over the ends of $Y_\pm$ as in \autoref{def ACyl G2}. We fix choices of ECyl submanifolds $P_{C,\pm}$ and ECyl tubular neighbourhood maps $\Upsilon_{P_C,\pm}$ for $P_\pm$, as described in \autoref{def P_C}. The notation $K_{P_{C,\pm}}$ is also retained from there.
   
\begin{definition}[Approximate associative submanifolds]\label{def approx asso tcs}
We define a closed $3$-dimensional submanifold $P_{T,C}$ of $Y_T$ by
$$P_{T,C}:=P_{T,C,+}\bigcup_F P_{T,C,-}$$
 where $P_{T,C,\pm}:=K_{P_{C,\pm}}\cup\Upsilon_\pm((T_0,T+1]\times \Sigma_\pm)$. Here the identification map $F:[T,T+1]\times Z_+\to [T,T+1]\times Z_-$ is given by
 $F(t,z)=(2T-t+1,f(z))$.  The normal bundle of $P_{T,C}$ is 
 $$NP_{T,C}=NP_{T,C,+}\bigcup_F NP_{T,C,-}.$$
The tubular neighbourhood map is defined by 
  $$\Upsilon_{P_{T,C}}:=\Upsilon_{P_{T,C,+}}\bigcup_F \Upsilon_{P_{T,C,-}}:V_{P_{T,C}}\to U_{P_{T,C}}$$ where $\Upsilon_{P_{T,C,\pm}}$ is the restriction of the ECyl tubular neighbourhood maps $\Upsilon_{P_{C,\pm}}$  on $P_{T,C,\pm}$. 
  
   Set $\Upsilon_{T}:=\Upsilon_{T,+}\cup_F \Upsilon_{T,-}$, where $\Upsilon_{T,\pm}$ is the restriction of $\Upsilon_{\pm}$ on the ends of $Y_{T,\pm}$. Let $\beta_\pm\in C^\infty(NP_{C,\pm})$ represent $P_\pm$ as in the above discussion satisfying $\Upsilon\circ\Upsilon_{P_{C,\pm}}(\beta_\pm)=P_\pm$. Set 
  \begin{equation} \label{eqn betaT}
  \beta_T:=\big(1-\chi_{T-1}\big)\beta_\pm\in C^\infty(NP_{T,C,\pm}).
  \end{equation}
   We define the approximate associative $P_T$ by
 \begin{equation}\label{eqn tcsPT}
 P_T:=\Upsilon_T\circ \Upsilon_{P_{T,C}}(\beta_T) \subset Y_T.
 \end{equation}
The diffeomorphism type of $P_T$ is independent of $T$, we call it the twisted connected sum $P_+ \#_f P_-$.  Finally, there is also a canonical bundle isomorphism $\Theta^{C}_{P_T}:NP_{T,C}\to NP_T$.
 \end{definition}
  
  Our goal is to deform $P_T$ into an associative submanifold $\widetilde{P}_T$ in the $G_2$-manifold $(Y_T, \phi_T)$. To achieve this, we define a non-linear map whose zero set corresponds to associative submanifolds in $(Y_T, \phi_T)$.

  \begin{definition}\label{def nonlinear map with T}
Define $\mathfrak F_T: C^{\infty}(V_{P_{T,C}})\to C^{\infty}{(NP_{T,C})}$ as follows: for all $u\in C^{\infty}(V_{P_{T,C}})$ and $w\in C^\infty (NP_{T,C})$, 
	 \begin{equation*}
	 	\inp{\mathfrak F_T(u)}{w}_{L^2}:=\int_{\Gamma_{u}}\iota_{w}\Upsilon_{P_{T,C}}^*\psi_T.
	 \end{equation*}
Here $\psi_T$ is the Hodge dual $4$-form of the $G_2$-structure $\phi_T$, and $\Gamma_u$ denotes the graph of $u$. The notation $w$ in the integrand is the fiberwise translation of $w\in C^\infty(NP_{T,C})$. The $L^2$ inner product used is defined via the canonical bundle isomorphism $\Theta_{P_T}^{C} : NP_{T,C} \to NP_T$ and the metric on $NP_T$ induced by the $G_2$-metric $g_{\phi_T}$. 
\end{definition}

\begin{definition}\label{def Linearization L_T}The \textbf{linearization} of $\mathfrak F_T$ at the $\beta_T$ from \autoref{eqn betaT} is 
\begin{equation*}
\bD_T:=d\mathfrak F_{T|_{\beta_T}}:C^{\infty}{(NP_{T,C})}\to C^{\infty}{(NP_{T,C})}.\qedhere
\end{equation*}
 \end{definition}
\begin{definition}\label{def decomposition of FT nonlinear map}
Given $T$, the \textbf{error} is $e_T:=\mathfrak F_T(\beta_T)\in C^{\infty}{(NP_{T,C})}$. Define the nonlinear map $Q_T: C^{\infty} (V_{P_{T,C}})\to C^{\infty}{(NP_{T,C})}$ such that 
$$\mathfrak F_T=\bD_{T}+ Q_T+e_T.$$  Note that $Q_T(\beta_T)=-\bD_{T}(\beta_T)$.
\end{definition}

\begin{notation}
	From now on, we will denote the Hölder spaces $C^{k,\gamma}(NP_{T,C})$ simply by $C^{k,\gamma}$, omitting the explicit dependence on $T$, unless otherwise specified. The relevant value of $T$ will always be clear from the context. The same convention applies to other Banach spaces.  Recall the asymptotic rates  $\mu_\pm$ and $\nu_\pm$ from the beginning of this section. Set 
	\begin{equation*}\mu:=\max\{\mu_+,\mu_-\}<0,\ \ \nu:=\max\{\nu_+,\nu_-\}<0.\qedhere\end{equation*}
\end{notation}

\begin{prop}[{\textbf{Error estimate}}]\label{prop error estimate with T}
  For all sufficiently large $T\geq T_0$ and for all $k\in \N\cup\{0\}$, $\gamma\in (0,1)$, we have 
  $$\norm{\mathfrak F_T(\beta_T)}_{C^{k,\gamma}}=\norm{e_T}_{C^{k,\gamma}}= O(e^{-\delta_e T}),$$
  where $\delta_e:=\min\{\delta,-\mu, -\nu\} $ and $\delta$ is the constant from \autoref{thm Kovalev tcs}.
\end{prop}
\begin{proof}Denote by $\phi_0$ the asymptotic $G_2$-structure on $\R^+\times Z_\pm$. Denote by $\psi_\pm$ and $\psi_0$ the Hodge dual of the $G_2$-structures $\phi_\pm$ and $\phi_0$, respectively. Set $A_{T,\pm}:=P_{T,C,\pm}\setminus P_{T-1,C,\pm}$. By \autoref{thm Kovalev tcs}, for all sufficiently large $T\geq T_0$, we have 
$$\abs{\Upsilon_{P_{T,C,\pm}}^*\psi_T-\psi_\pm}=O(e^{-\delta T}) \quad \text{over} \quad P_{T-1,C,\pm} ,$$  
and 
$$\abs{\Upsilon_{P_{T,C,\pm}}^*\psi_T-\Upsilon_*\psi_0}=O(e^{-\delta T})+O(e^{\nu T}) \quad \text{over} \quad A_{T,\pm} .$$  
 Since $P_\pm$ are associative submanifolds with respect to $\phi_\pm$, for all $w\in C^{\infty}(NP_{T,C})$, we have 
$$\inp{\mathfrak F_T(\beta_T)}{w}_{L^2}\leq(O(e^{-\delta T}+e^{\nu T})\norm{w}_{L^2}+\int_{A_{T,+}}\iota_{w}\Upsilon_*\psi_++\int_{A_{T,-}}\iota_{w}\Upsilon_*\psi_-.$$
Thus, $$\abs{\mathfrak F_T(\beta_T)}\lesssim O(e^{-\delta T}+e^{\nu T}+ \max \norm {\beta_T-\beta_\pm}_{C^1({A_{T,\pm}})})= O(e^{-\delta_e T}).$$ 
Estimates for higher derivatives and Hölder norms of $\mathfrak F_T(\beta_T)$ also hold similarly.
\end{proof}

\begin{prop}[{\textbf{Quadratic estimate}}]\label{prop quadratic estimate with T}There is a constant $C>0$ such that for all sufficiently large $T\geq T_0$ and for all $u, v\in C^{\infty}(V_{P_{T,C}})$, $\eta \in C^{\infty}(NP_{T,C})$, we have
	\begin{enumerate}[label=(\roman*), leftmargin=*]	
\item $\abs{d\mathfrak F_{T|_{u}}(\eta)-d\mathfrak F_{T|_{v}}(\eta)}\leq C (\abs{u-v}+\abs{\nabla^\perp (u-v)})(\abs{\eta}+\abs{\nabla^\perp \eta})$, 
\item $\abs{Q_T(u)-Q_T(v)}\leq C (\abs{u-v}+\abs{\nabla^\perp (u-v)})(\abs{u-\beta_T}+\abs{\nabla^\perp (u-\beta_T)}+\abs{v-\beta_T}+\abs{\nabla^\perp (v-\beta_T)})$,
\item $\norm{Q_T(u)-Q_T(v)}_{C^{k,\gamma}}\leq C\norm{u-v}_{C^{k+1,\gamma}}(\norm{u-\beta_T}_{C^{k+1,\gamma}}+\norm{v-\beta_T}_{C^{k+1,\gamma}}), \forall k\in \N\cup\{0\}, \gamma\in (0,1).$
\end{enumerate}
\end{prop}

To proceed with the proof, we make use of the following lemma, the proof of which can be found in \cite[Lemma A.1]{Bera2023}.

\begin{lemma}\label{lem quadratic}There is a constant $C>0$ such that for all sufficiently large $T\geq T_0$ and for all $u, v,s \in C^{\infty}(V_{P_{T,C}})$, $w \in C^{\infty}(NP_{T,C})$, over $\Gamma_s:=\graph s\subset V_{P_{T,C}}$ we have
 $$\abs{\iota_w\mathcal L_u\mathcal L_v(\Upsilon_{P_{T,C}}^*\psi_T)}\leq C \abs{w}  (\abs{u}+\abs{\nabla^\perp u})(\abs{v}+\abs{\nabla^\perp v}).$$
\end{lemma}

\begin{proof}[{\normalfont{\textbf{Proof of \autoref{prop quadratic estimate with T}}}}] 
For all $w\in C_c^\infty(NP_{T,C})$  we compute:
\begin{align*}\inp{d\mathfrak F_{T|_{u}}(\eta)-d\mathfrak F_{T|_{v}}(\eta)}{w}_{L^2}
&=\int_0^1\Big(\frac d{dt}\int_{\Gamma_{tu+(1-t)v}} \cL_{\eta}\iota_w(\Upsilon_{P_{T,C}}^*\psi_T)\Big)dt\\
&=\int_0^1\int_{\Gamma_{tu+(1-t)v}} \cL_{(u-v)}\cL_{\eta}\iota_w(\Upsilon_{P_{T,C}}^*\psi_T)dt.
\end{align*}
Since $u, v,w$ and $\eta$ in the integrand are fiberwise translations,  $[u-v, w]=0$ and $[\eta,w]=0$, and therefore the last expression is same as 
$$\int_0^1\int_{\Gamma_{tu+(1-t)v}} \iota_w \cL_{(u-v)} \cL_{\eta}(\Upsilon_{P_{T,C}}^*\psi_T)dt.$$
The required estimate in (i) now follows from \autoref{lem quadratic}. To see (ii) we compute:
\[Q_T(u)-Q_T(v) =\int_0^1dQ_{T|_{tu+(1-t)v}}(u-v)dt
 =\int_0^1\big(d\mathfrak F_{T|_{tu+(1-t)v}}(u-v)-d\mathfrak F_{T|_{\beta_T}}(u-v)\big)dt.\]
Then (ii) follows from (i).  The estimates in (iii) with the H\"older norms follow by the same kind of argument. 
\end{proof}

The rest of the subsection deals with the linearization $\bD_T$ defined in \autoref{def Linearization L_T}.

\begin{prop}\label{prop self-adjoint T}
For all sufficiently large $T\geq T_0$  the linearization $\bD_T$ is a formally self-adjoint elliptic operator . 
\end{prop}
\begin{proof}For all $v,w\in C^{\infty}(N{P_{T,C}})$,
\begin{align*}
\inp{\bD_{T}v}{w}_{L^2}-\inp{w}{\bD_{T}v}_{L^2}&=\int_{\Gamma_{\beta_T}}\mathcal L_v\iota_w(\Upsilon_{P_{T,C}}^*\psi_T)-\mathcal L_w\iota_v(\Upsilon_{P_{T,C}}^*\psi_T)\\
&=\int_{\Gamma_{\beta_T}}\iota_w \iota_v(\Upsilon_{P_{T,C}}^*d\psi_T)+{\iota_{[v,w]}(\Upsilon_{P_{T,C}}^*\psi_T)}.
\end{align*}
Since $\phi_T$ is a torsion free $G_2$-structure, $d\psi_T=0$. Since $v$ and $w$ in the integrand are fiberwise translations, $[v,w]=0$. Hence $\bD_{T}$ is formally self-adjoint.

 It remains to prove that $\bD_T$ is an elliptic operator for all sufficiently large $T\geq T_0$. We denote the restrictions of $\bD_{T}$ and $\beta_T$ over $P_{T,C,\pm}$ by $\bD_{T,\pm}$ and $\beta_{T,\pm}$, respectively. 
By \autoref{thm Kovalev tcs} $\nabla^k(\phi_T-\phi_\pm)=O(e^{-\delta T})$ with $\delta>0$ and by definition $\nabla^k(\beta_T-\beta_\pm)=O(e^{\mu T})$ for all $k\in \N\cup \{0\}$.  Therefore, by \autoref{prop quadratic estimate with T}(i) and a similar kind of argument presented in the proof of \autoref{prop error estimate with T} implies (under appropriate canonical bundle isomorphisms) : for all $u\in C^{\infty}{(NP_{T,C,\pm})}$
\begin{equation}\label{eq comparison of linearizations}
\bD_{T,\pm}u=\bD_{P_\pm} u+O(e^{\max\{\mu,-\delta\} T})(\abs{u}+\abs{\nabla^\perp u}),
\end{equation}
 where $\bD_{P_\pm}$ are Fueter operators defined in \autoref{eq Fueter operator}. Hence, for all sufficiently large $T\geq T_0$, $\bD_T$ is an elliptic operator.
\end{proof}

\begin{prop}[{\textbf{Schauder estimate}}]\label{prop Schauder estimate}There exists a constant $C>0$ such that for all sufficiently large $T\geq T_0$,  $k\in \N\cup\{0\}$, $\gamma\in (0,1)$ and for all  $u\in C^{k+1,\gamma}$, we have		
$$\norm{u}_{C^{k+1,\gamma}}\leq C (\norm{\bD_{T}u}_{C^{k,\gamma}}+\norm{u}_{L^\infty}).$$
\end{prop}
\begin{proof}For $u\in C^{k+1,\gamma}$, define $u_\pm\in C^{k+1,\gamma}({NP_{T,C,\pm}})$ by restricting $u$ over $P_{T,C,\pm}$. Using interior Schauder estimates for $\bD_{P_\pm}$ applied to $u_\pm$ we get
$$\norm{u}_{C^{k+1,\gamma}}\lesssim \norm{\bD_{P_+}u_+}_{C^{k,\gamma}}+\norm{\bD_{P_-}u_-}_{C^{k,\gamma}}+\norm{u}_{L^\infty},$$
 where $\bD_{P_\pm}$ are the Fueter operators defined in \autoref{eq Fueter operator}. Similarly to \autoref{eq comparison of linearizations} we obtain 
\begin{align*}
	\norm{\bD_{T}u_\pm-\bD_{P_\pm}u_\pm}_{C^{k,\gamma}}&\lesssim \norm{u_\pm}_{C^{k+1,\gamma}}\norm{\beta_T-\beta_\pm}_{C^{k+1,\gamma}}+O(e^{-\delta T})\norm{u}_{C^{k+1,\gamma}}\\
	&\lesssim (O(e^{-\delta T})+O(e^{\mu T}))\norm{u}_{C^{k+1,\gamma}}.
\end{align*}
Here $\delta$ is the constant from \autoref{thm Kovalev tcs}. Hence, we get the required estimate.
\end{proof}
The final estimate we require is a uniform linear estimate for $\bD_T$ valid for all sufficiently large $T$, providing a positive lower bound depending only on $T$. However, elements in $\ker \bD_{P_\pm,0}$ whose asymptotic limits match via $f$ form an approximate kernel where such an estimate cannot hold. We therefore restrict to its complement, analogous to \cite[Theorem 3.24]{SaEarp2013}. This is precisely the reason for imposing \autoref{hyp gluing}, which is equivalent to requiring that the approximate kernel vanishes. Under this assumption, the desired lower bound holds on all of $C^{k+1,\gamma}$. Then \autoref{prop self-adjoint T} implies that $\bD_T$ is invertible, with a uniformly bounded inverse depending only on $T$, which suffices for the proof of the gluing theorem.

\begin{definition}\label{def approx kernel of L_T} 
We define	
\begin{enumerate}[label=(\roman*), leftmargin=*]
\item the \textbf{matching kernel} $\mathcal K^{\mathfrak m}$ by
$$\mathcal K^{\mathfrak m}:=\{(u_{+},u_{-})\in\ker\bD_{P_+,0}\times \ker\bD_{P_-,0}:f_*\iota^+_\infty u_{+}=\iota^-_\infty u_{-}\},$$
\item the \textbf{approximate kernel} of $\bD_{T}$ by
$$\mathcal K^{\mathfrak m}_{T}:=\{u_+\#_Tu_-\in C^\infty(NP_{T,C}):(u_{+},u_{-})\in\mathcal K^{\mathfrak m}\},$$
where $ u_+\#_Tu_-$ over $P_{T,C,\pm}$ is
 $$ u_+\#_Tu_-:=u_\pm-\chi_{T-1}(u_\pm-u_\mp),$$
\item the \textbf{complement of the approximate kernel} by 
$$\mathcal X_{T}^{k+1,\gamma}:=\{u\in C^{k+1,\gamma}:\inp{u}{\xi}_{L^2(K_{P_\pm})}=0, \forall  \xi\in\mathcal K^{\mathfrak m}_{T}\},$$
where $K_{P_\pm}$ are the compact submanifolds of $P_{\pm}$ with boundary from \autoref{def ACyl asso}. Since by \autoref{lem fourier series Donaldson Acyl}(ii) $\ker \bD_{P_\pm,0}$ are finite dimensional, unique continuation theorem implies: $C^{k+1,\gamma}=\mathcal K^{\mathfrak m}_{T} \oplus \mathcal X_{T}^{k+1,\gamma}.$ \qedhere
\end{enumerate}
\end{definition}
The following proposition provides the desired linear estimate. We include the proof, following the same line of argument as in \cite[Theorem 3.24]{SaEarp2013}, but with additional detail, as this estimate plays a central role in the proof of the gluing theorem.

\begin{prop}[{\textbf{Linear estimate}}]\label{prop linear estimate with T} For all $\epsilon>0$ there exists $T_\epsilon\geq T_0$ such that for all $T\geq T_\epsilon$, $k\in \N\cup\{0\}$, $\gamma\in (0,1)$ and for all $u\in \mathcal X_{T}^{k+1,\gamma}$,
	we have
	$$\norm{u}_{C^{k+1,\gamma}}\lesssim e^{\epsilon T}\norm{\bD_{T}u}_{C^{k,\gamma}}.$$  
\end{prop}
\begin{proof} By the Schauder estimate in \autoref{prop Schauder estimate}, we see it is enough to prove that for all sufficiently large $T\geq T_0$ and for  all $u\in\mathcal X_{T}^{k+1,\gamma}$,
	we have
	$$\norm{u}_{L^\infty}\lesssim e^{\epsilon T}\norm{\bD_{T}u}_{C^{k,\gamma}}.$$
	We will prove this by contradiction. Suppose this is not true, then there exists an increasing sequence $T_n\geq T_0$ which tends to $\infty$ as $n\to\infty$ and $u_n$ in $\mathcal X_{T_n}^{k+1,\gamma}$  such that 
	$$\norm{u_n}_{L^\infty}=1,\ 
	 e^{\epsilon T_n}\norm{\bD_{T_n}u_n}_{C^{k,\gamma}}\to 0 \ \text{as}\ n\to\infty.$$
	Define $u_{n,\pm}\in C^{k+1,\gamma}({NP_{T_n,C,\pm}})=:C^{k+1,\gamma}_{P_{T_n,C,\pm}}$ by restricting $u_n$ over $P_{T_n,C,\pm}$.
	Again by the Schauder estimate in \autoref{prop Schauder estimate}, we see that $\norm{u_n}_{C^{k+1,\gamma}}$ is bounded and hence  $\norm{u_{n,\pm}}_{C^{k+1,\gamma}}$ are also bounded.  A moment's thought shows that by the Arzelà-Ascoli theorem, there exists a subsequence which we call again $u_n$, and there exist $u_\pm$ in $C^{k+1,\gamma/ 2}_{P_\pm,0}$ such that $\bD_{P_\pm}u_\pm=0$ and $u_{n,\pm}\to u_\pm$ in $C^{k+1,\gamma/2}_{P_\pm, \operatorname{loc}}$. Moreover, by the elliptic regularity \cite[Proposition 4.14]{Bera2025}, we get $u_\pm\in  C^{k+1,\gamma}_{P_\pm,0}$ and therefore $u_\pm\in \ker \bD_{P_\pm,0}$. By taking further subsequences if necessary we claim that as $n\to\infty$,
 \begin{equation}\label{eq:main linear estimate}
 \norm{ u_{n,\pm}-u_{\pm}}_{L^\infty(NP_{T_n,C,\pm})}\leq \norm{ u_{n,\pm}-u_{\pm}}_{C^{k+1,\gamma}_{P_{T_n,C,\pm}}}\to 0.
 \end{equation}
To prove this claim we argue as follows. Set $\widetilde u_{n,\pm}:= (1-\chi_{\frac 32 T_n})u_n \in C^{k+1,\gamma}_{P_\pm}$. Again by \autoref{eq comparison of linearizations}, we obtain that for any sufficiently small $0<s<\epsilon$ we have 
	\begin{align*}
	\norm{\bD_{P_\pm}\widetilde u_{n,\pm}}_{C^{k,\gamma}_{P_\pm,s}}&\lesssim e^{-(\delta+s) T_n}\norm{u_n}_{C^{k+1,\gamma}}+e^{(\mu-s) T_n}\norm{u_n}_{C^{k+1,\gamma}}+\norm{\bD_{T_n}\widetilde u_{n,\pm}}_{C^{k,\gamma}}\\
	&\lesssim (e^{-(\delta+s) T_n}+e^{(\mu-s) T_n}+e^{-\frac 32s T_n})\norm{u_n}_{C^{k+1,\gamma}}+\norm{\bD_{T_n} u_{n}}_{C^{k,\gamma}}\\
		&\lesssim e^{-(\delta+s) T_n}+e^{(\mu-s) T_n}+e^{-\frac 32s T_n}+ e^{-\epsilon T_n}.
	\end{align*}	 
By \autoref{lem fourier series Donaldson Acyl}(ii) and \autoref{prop main Fredholm acyl}, there exists $v_{n,\pm}\in \ker\bD_{P_\pm,s}=\ker\bD_{P_\pm,0}$ such that 
$$\norm{\widetilde u_{n,\pm}-v_{n,\pm}}_{C^{k+1,\gamma}_{P_\pm,s}}\lesssim  e^{-(\delta+s) T_n}+e^{(\mu-s) T_n}+e^{-\frac 32s T_n}+ e^{-\epsilon T_n}.$$
Hence as  $n\to\infty,$
$$\norm{u_{n,\pm}-v_{n,\pm}}_{C^{k+1,\gamma}_{P_{T_n,C,\pm}}}\lesssim  e^{-\delta T_n}+e^{\mu T_n}+e^{-\frac s2 T_n}+ e^{(s-\epsilon)T_n}
\to 0.$$ 
As $\ker\bD_{P_\pm,0}$ is finite dimensional, the norms $\norm{\cdot}_{C^{k+1,\gamma /2}_{K_{P\pm}}}$ and $\norm{\cdot}_{C^{k+1,\gamma}_{P_\pm}}$ are equivalent on it. Taking further subsequence yields $\norm{ v_{n,\pm}-u_\pm}_{C^{k+1,\gamma/2}_{K_{P\pm}}}\to 0$ and hence as $n\to \infty$, $\norm{ v_{n,\pm}-u_\pm}_{C^{k+1,\gamma}_{P_\pm}}\to 0 $. This proves \autoref{eq:main linear estimate}.

 Moreover \autoref{eq:main linear estimate} implies that $f_*\iota^+_\infty u_{+}=\iota^-_\infty u_{-}$, that is $(u_{+},u_{-})\in\mathcal K^{\mathfrak m}$. Indeed, for all $z\in \Sigma_+$,
$$f_*\iota^+_\infty u_{+}(z)=\lim_{n\to \infty}F_*u_{n,+}(T_n,z)=\lim_{n\to \infty}u_{n,-}(T_n,f(z))=\iota^-_\infty u_{-}(f(z)).$$
As $u_n\in \mathcal X_{T_n}^{k+1,\gamma}$, we have 
$$\norm{u_\pm}_{L^2(K_{P_\pm})}=\inp{u_\pm-u_{n,\pm}}{u_\pm}_{L^2(K_{P_\pm})}\leq \norm{u_{n_\pm}-u_\pm}_{L^\infty(K_{P_\pm})}\norm{u_\pm}_{L^1(K_{P_\pm})}\to 0.$$
Since $u_\pm\in \ker \bD_{P_\pm,0}$, by unique continuation, $u_\pm=0$, which is a contradiction because as $n\to \infty$, 
 \begin{equation*}
 	1=\norm{u_n}_{L^\infty}\leq \norm{ u_{n,+}}_{L^\infty(NP_{T_n,C,+})}+\norm{ u_{n,-}}_{L^\infty(NP_{T_n,C,-})}\to 0.\qedhere
 	 \end{equation*}
\end{proof}

\subsection{Proof of the gluing theorem}\label{subsec Proof of the gluing theorem}
To prove the gluing theorem, we solve the nonlinear equation $\mathfrak{F}_T u = 0$ for $u$ sufficiently close to $\beta_T$, using the following application of the Banach contraction principle \cite[Lemma 7.2.23]{Donaldson1990}.

 \begin{lemma}\label{lem Banach contraction acyl}
	Let $\mathcal X,\mathcal Y$ be two Banach spaces and let $V\subset \mathcal X$ be a neighbourhood of $0\in \mathcal X$. Let $x_0\in V$. Let $F:V\to \mathcal Y$ be a smooth map of the form
	$$F(x)=L(x)+Q(x)+F(x_0)   \quad \big(\text{so}\ \ Q(x_0)=-L(x_0)\big)$$
	such that:
	\begin{itemize}
		\item $L:\mathcal X\to \mathcal Y$ is a linear bounded invertible operator and there exists a constant $c_L>0$ such that for all $x\in \mathcal X$, $\norm{x}_{\mathcal X}\leq c_L\norm{Lx}_{\mathcal Y}.$
		\item $Q:V\to \mathcal Y$ is a smooth map and there exists a constant $c_Q>0$ such that for all $x_1,x_2\in V$, 
		$$\norm{Q(x_1)-Q(x_2)}_\mathcal Y\leq c_Q\norm{x_1-x_2}_\mathcal X(\norm{x_1-x_0}_\mathcal X+\norm{x_2-x_0}_\mathcal X).$$
			\end{itemize}
			
	If  $\norm{F(x_0)}_{\mathcal Y}\leq\frac 1{10 c_L^2c_Q}$ and $B(x_0,\frac 1{5 c_Lc_Q})\subset V$, then there exists a unique $x\in \mathcal X$ with $\norm{x-x_0}_{\mathcal X}\leq\frac 1{5 c_Lc_Q}$ solving $F(x)=0$.
\end{lemma}

\begin{proof}[{{{\normalfont{\textbf{Proof of \autoref{thm gluing associative submanifolds}}}}}}]

The nonlinear map $\mathfrak F_T: C^{k+1,\gamma} (V_{P_{T,C}})\to C^{k,\gamma}{(NP_{T,C})}$ has been expressed in \autoref{def nonlinear map with T} as 
 $$\mathfrak F_T=\bD_{T}+ Q_T+\mathfrak F_T(\beta_T).$$ 
The \autoref{hyp gluing} implies that the matching kernel $\mathcal K^{\mathfrak m}=0$. Therefore by \autoref{prop linear estimate with T}, there exists $T_0^\prime\geq T_0$ such that for all $T\geq T_0^\prime$, we have $\bD_T:C^{k+1,\gamma}{(NP_{T,C})}\to C^{k,\gamma}{(NP_{T,C})}$ is an invertible operator and there exists a positive constant $c_{\bD_{T}}=O(e^{\frac{\delta_e}4T})$ such that for all $u\in C^{k+1,\gamma}$,
	we have
	$$\norm{u}_{C^{k+1,\gamma}}\leq c_{\bD_{T}} \norm{\bD_{T}u}_{C^{k,\gamma}}.$$ 
Here $\delta_e$ is the constant from	\autoref{prop error estimate with T}. By \autoref{prop quadratic estimate with T} there exists a positive constant $c_{Q_T}=O(1)$ such that for all $u,v \in C^{k+1,\gamma} (V_{P_{T,C}})$ we have
	$$\norm{Q_T(u)-Q_T(v)}_{C^{k,\gamma}}\leq c_{Q_T}\norm{u-v}_{C^{k+1,\gamma}}(\norm{u-\beta_T}_{C^{k+1,\gamma}}+\norm{v-\beta_T}_{C^{k+1,\gamma}}).$$
Therefore by \autoref{prop error estimate with T}, there exists $T_0^{\prime\prime}\geq T_0^\prime$ such that for all $T\geq T_0^{\prime\prime}$, we have $$\norm{\mathfrak F_T(\beta_T)}_{C^{k,\gamma}}=O(e^{-\delta_eT})\leq \frac 1{10 c_{\bD_{T}}^2c_{Q_T}}$$
and hence by \autoref{lem Banach contraction acyl}, there exists a unique family $\{\widetilde{\beta}_T\in C^{{k+1,\gamma}} (V_{P_{T,C}}):T\geq T_0^{\prime\prime}\}$ such that $\mathfrak F_T(\widetilde{\beta}_T)=0$ and 
\begin{equation}\label{eq gluing thm close}
	\norm{\widetilde{\beta}_T-{\beta}_T}_{C^{k+1,\gamma}}=O(e^{-\frac{\delta_e}4T}).
\end{equation}
 Thus we obtain a family of closed associative submanifolds $\widetilde P_T:=\Upsilon_{P_{T,C}}(\widetilde{\beta}_T)$.  
  
  It remains to prove that $\widetilde P_T$ is rigid, that is $d\mathfrak F_{T|_{\widetilde{\beta}_T}}$ is invertible. 
  We choose $\epsilon:=\frac{\delta_e}8$. Then by \autoref{prop linear estimate with T}, there exists $T_\epsilon>0$ such that for all $T\geq T_\epsilon$, $\norm {\bD_T^{-1}}=O(e^{\epsilon T})$. Now by \autoref{prop quadratic estimate with T} and \autoref{eq gluing thm close},
$$\norm {\bD_T^{-1}d\mathfrak F_{T|_{\widetilde{\beta}_T}}-\id}\leq \norm {\bD_T^{-1}}\norm {d\mathfrak F_{T|_{\widetilde{\beta}_T}}-\bD_T}=O\big(e^{(\epsilon-\frac{\delta_e}4)T}\big)=O(e^{-\frac{\delta_e}8T}).$$ 
Therefore by defining $T_1:= \max{\{T_\epsilon,T_0^{\prime\prime}\}}$, we obtain the required theorem.
 \end{proof}

\section{Associative submanifolds by gluing ACyl holomorphic curves}\label{sec Associative submanifolds by gluing of ACyl holomorphic curves}
In this section, we rephrase \autoref{hyp gluing} in the setting where the ACyl associative submanifolds are products of ACyl holomorphic curves with the unit circle as in \autoref{eg ACyl asso from ACyl holo}(i). In particular, using \autoref{thm gluing associative submanifolds}, we construct associative submanifolds in the twisted connected sum $G_2$-manifolds defined in \autoref{def TCS G2}, arising from pairs of building blocks with a hyperkähler rotation (see \autoref{section TCS}). This is formalized in the following theorem.
 
 	\begin{theorem}\label{thm gluing asso from ACyl holo}
	Let $(Z_\pm,X_\pm,\pmb \omega_\pm)$ be a pair of framed building blocks with a hyperk\"ahler rotation $\mathfrak r:X_+\to X_-$. Let $V_\pm:=Z_\pm\setminus X_\pm$ be a pair of corresponding ACyl Calabi--Yau $3$-folds and let $\{(Y_T,\phi_T)\}$ be the family of $G_2$-manifolds obtained from the twisted connected sum construction. Let ${\sC}_\pm$ be a pair of embedded holomorphic curves in $Z_\pm$ intersecting $X_\pm$ transversely at $\bar x_\pm:=\{x_{1,\pm},x_{2,\pm},...,x_{m,\pm}\}$. Then ${\sC}^*_\pm:={\sC}_\pm\setminus \bar x_\pm$ are ACyl embedded holomorphic curves in $V_\pm$ with asymptotic cross sections $\displaystyle \amalg_{j=1}^mS^1\times \{x_{j,\pm}\}\subset S^1\times X_\pm$. Assume that  
	\begin{itemize} 
	\item $\mathfrak r(\bar x_{+})=\bar x_{-}$,
	\item $H^0( {\sC}_\pm,N {\sC}_\pm(-\bar x_\pm))=0$, that is, there are no holomorphic normal vector fields of $\sC_\pm$ vanishing at $\bar x_\pm$, 
	\item $\operatorname{im}(\mathfrak r_*\operatorname{ev}_+)\ \cap\  \operatorname{im}({\operatorname{ev}_-})=\{0\}$, where $$\operatorname{ev}_\pm:=\bigoplus_{j=1}^m\operatorname{ev}_{x_{j,\pm}}:H^0({\sC}_\pm,N{\sC}_\pm)\to   \bigoplus_{j=1}^m T_{x_{j,\pm}}X_\pm$$
are the evaluations maps.
	\end{itemize}
Then the pair of ACyl associative submanifolds $S^1\times {\sC}^*_\pm$ satisfies \autoref{hyp gluing} and there is a family of closed rigid associative submanifolds $\widetilde P_T$ in $(Y_T,\phi_T)$ for all sufficiently large $T$. Moreover, these are diffeomorphic to the twisted connected sum $(S^1\times  {\sC}^*_+) \#_\tau (S^1\times {\sC}^*_-)$ along the $m$ tori, where $\tau$ swaps the circles on each asymptotic cross section.  
\end{theorem}
The proof of the above theorem is based on the following observations.
\begin{lemma}\label{Lem asymptotic ACyl asso from ACyl holo}
Let $Y:=S^1\times V$ be an ACyl $G_2$-manifold with asymptotic cross section $T^2\times X$ of the form described in  \autoref{rmk product ACyl G_2}. Let $\sC^*$ be an ACyl embedded holomorphic curve in $V$ with asymptotic cross section $\amalg_{j=1}^mS^1\times \{x_j\}\subset S^1\times X$. Consider the ACyl associative submanifold $P:=S^1\times \sC^*$ with  asymptotic cross section $\Sigma:=\amalg_{j=1}^m T^2\times \{x_j\}$. Then there are canonical isomorphsims: 
\[\ker{\mathbf D}_{\Sigma} \cong \bigoplus_{j=1}^m T_{x_j}X \qandq \ker{\bD}_{P,0} \cong \ker \mathbf D_{\sC^*,0} \]
where $\ker \mathbf D_{\sC^*,0}$ is the space of all bounded holomorphic normal vector fields of $\sC^*$.
\end{lemma}

\begin{proof} Since the normal bundle of each $T^2 \times \{x_j\} \subset T^2 \times X$ is the flat trivial bundle with fiber $T_{x_j}X$, there is a canonical isomorphsim
$
\ker{\mathbf D}_{\Sigma} = \bigoplus_{j=1}^m \ker D_{T^2 \times \{x_j\}} \cong \bigoplus_{j=1}^m T_{x_j}X,
$
as each kernel in the direct sum consists of covariantly constant sections.

By pulling back we have an inclusion map $\ker \mathbf D_{\sC^*,0}\hookrightarrow \ker{\bD}_{P,0}$. We claim that this is an isomorphism. Since $\mathbf D_{\sC^*}$ is ${J}$-antilinear, formally self-adjoint and commutes with $\partial_\theta$ (recall $\theta$ denotes the coordinate of $S^1$ factor of $Y$), therefore integration by parts shows that any normal vector field $u \in  \ker{\bD}_{P,0}$ satisfies (as the asymptotic limit is covariantly constant):
 $$\inp{\mathbf D_{\sC^*}u}{{J}\partial_\theta u}_{L^2(N\sC^*)}=0.$$ 
Therefore by the identity: $\mathbf D_P= J\partial_\theta+ \mathbf D_{\sC^*}$, such $u$ satisfies $\partial_\theta u=0$ and $\mathbf D_{\sC^*}u=0$. This proves the claim.
\end{proof}

\begin{lemma}\label{Lem ACyl asso from holo in building blocks}
	Let $(Z,X,\pmb \omega)$ be a framed building block and $V:=Z\setminus X$ be the corresponding ACyl Calabi--Yau $3$-fold. Let ${\sC}$ be an embedded holomorphic curve in $Z$ intersecting $X$ transversely at $\bar x:=\{x_1,x_2,...,x_m\}$. Then $\sC^*:={\sC}\setminus \bar x $ is an ACyl embedded holomorphic curve in $V$. Moreover, there is a canonical isomorphism
$${\Lambda_0}:\ker \mathbf D_{\sC^*,0} \to H^0({\sC},N{\sC})$$
satisfying:
\[ \bigoplus_{j=1}^m\operatorname{ev}_{x_j} \circ {\Lambda_0}= \iota_\infty.  \]	
Here $H^0({\sC},N{\sC})$ is the space of all holomorphic normal vector fields of $\sC$. Moreover, $$\ker \iota_\infty\cong H^0( \sC,N \sC(- \bar x)),$$ the space of all holomorphic normal vector fields of $\sC$ vanishing at $\bar x$.
\end{lemma}
\begin{proof}In \cite{Haskins2012}, the definition of an ACyl Calabi--Yau $3$-fold has a diffeomorphism $\Upsilon: (T,\infty)\times S^1 \times X \to V\setminus K_V$ for some compact submanifold with boundary $K_V\subset V$ over the end. This is given by restricting a smooth embedding $\widetilde \Upsilon:\Delta \times X \to Z$, (where $\Delta$ is an open disc in $\C$ around $0$) pre-composed by the biholomorphism  $(T,\infty)\times S^1 \times X\to \Delta \times X$ taking $(t,z,x)\to (\frac{e^{-t}}z,x)$. Denote the coordinate for $\Delta$ by $w$. As $ \sC$  intersects $X$ transversely at $\bar x:=\{x_1,x_2,...,x_m\}$,  near each $x_j$ it can be written as $\widetilde\Upsilon_*(\Gamma_{\widetilde{\alpha}_j})$, where $\Gamma_{\widetilde{\alpha_j}}$ is the graph of a map $\widetilde{\alpha_j}:\Delta \times \{x_j\}\to T_{x_j} X $ (choosing $\Delta$ to be a sufficiently small disc) satisfying
 $$\abs{(\nabla^\perp)^k\widetilde{\alpha_j}}=O(\abs{w}^{1-k})$$
		  for all $j=1,2...,m$, $k\in \N\cup \{0\}$ as $w\to0$. Therefore $\sC^* $ is an ACyl holomorphic curve in $V$ with rate $-1$.
		  
	We will now prove that for each $-1\leq\lambda\leq 0$, there is a canonical injective linear map $$\Lambda_\lambda:\ker \mathbf D_{\sC^*,\lambda}\to H^0( \sC,N \sC(\left \lfloor{\lambda}\right \rfloor  \bar x)).$$
	 As $\widetilde \Upsilon^* J-J_{\C\times X}=0$ along $\Delta\times \{x\}$ for all $x\in X$ (see \cite[Appendix A]{Haskins2012}), $\widetilde w:=\widetilde \Upsilon_* w$ is a holomorphic function around $\bar x$ in $\sC$. Now, given $u\in \ker \mathbf D_{\sC^*,\lambda}$, we have ${\widetilde w}^{-\left \lfloor{\lambda}\right \rfloor } u$ is a bounded weak holomorphic section around $\bar x$ in $\sC$ and therefore it can be extended as a holomorphic section. Hence, $u$ can be extended uniquely to get an element $\Lambda_\lambda(u)$ in $ H^0(\sC,N\sC(\left \lfloor{\lambda}\right \rfloor  \bar x))$. 
	 
	 Moreover, $\Lambda_\lambda$ is surjective if $\lambda\in \Z^m$. In our case, as the asymptotic cross sections are flat tori with flat normal bundles, there are no critical rates in $[-1,0)$ \cite[Example 4.6]{Bera2025}.  Therefore, the same line of argument as in the proof of \autoref{prop main Fredholm acyl} implies that $\ker \iota_\infty=\ker \mathbf D_{\sC^*,-1}$. This completes the proof of the proposition.
\end{proof}
\begin{proof}[{{{\normalfont{\textbf{Proof of \autoref{thm gluing asso from ACyl holo}}}}}}] The theorem is immediate from \autoref{thm gluing associative submanifolds} using  \autoref{Lem asymptotic ACyl asso from ACyl holo} and \autoref{Lem ACyl asso from holo in building blocks}.
\end{proof}

\begin{remark}\label{rmk H1}Let $(Z,X)$ be a building block and $\ell\cong \P^1$ be a rational curve in $Z$ intersecting $X$ transversely at  $\bar x:=\{x_1,x_2,...,x_m\}$, $m\geq 1$. By Grothendieck's lemma, for some $k_1, k_2\in \Z$
 $$N\ell\cong\mathcal O_{\P^1}(k_1)\oplus \mathcal O_{\P^1}(k_2).$$ 
 As $\ell\cdot c_1(Z)=\ell\cdot [X]=m$, we have $k_1+k_2=\ell \cdot c_1(N\ell)=\ell \cdot \big([X]-c_1(T\ell)\big)=m-2$. Thus for some $k\in \Z$,
$$N\ell\cong \mathcal O_{\P^1}(k)\oplus \mathcal O_{\P^1}(m-k-2).$$

Therefore, if we assume $H^0(\ell,N\ell)= 0$, then $\ell$ does not intersect $X$ and $S^1\times \ell $ will be a closed rigid associative in $S^1\times V$ and $N\ell\cong \mathcal O_{\P^1}(-1)\oplus \mathcal O_{\P^1}(-1)$. 

We see that
\begin{align*}
	H^1(\ell,N\ell)&=H^1(\P^1,\mathcal O_{\P^1}(k))\oplus H^1(\P^1,\mathcal O_{\P^1}(m-k-2))\\
	&=H^0(\P^1,\mathcal O_{\P^1}(-k-2))\oplus H^0(\P^1,\mathcal O_{\P^1}(k-m))\ \ (\text{by Serre duality}),
\end{align*} 
 and
   $$H^0(\ell,N\ell(-\bar x))=H^0(\ell,N\ell\tn \mathcal O_{\P^1}(-m))=H^0(\P^1,\mathcal O_{\P^1}(k-m))\oplus H^0(\P^1,\mathcal O_{\P^1}(-k-2)).$$
Thus we obtain that $H^0(\ell,N\ell(-\bar x))\cong H^1(\ell,N\ell)=0$ if and only if $-1\leq k\leq m-1$. 
\end{remark}	

 As an application of \autoref{thm gluing asso from ACyl holo} we prove the next proposition that will help us to produce examples of associative submanifolds in the twisted connected sum $G_2$-manifolds. In particular, we construct associative $3$-spheres in many twisted connected sum $G_2$-manifolds arising from Fano $3$-folds (see \autoref{eg associative $3$-sphere}).	
 
\begin{prop}\label{prop associative sphere} Let $(Z_+, X_+)$ be a building block and let $\ell_+$ be an unobstructed holomorphic line in $Z_+$ (that is, $H^1(\ell_+,N\ell_+)=0$) intersecting $X_+$ transversely at a point. 
Let $W_-$ be a semi-Fano $3$-fold and let $X_- \in \abs{-K_{W_-}}$ be a smooth $K3$ surface in $W_-$ such that ${-K_{W_-}}|_{_{X_-}}$ is very ample and there exists a hyperkähler rotation $\mathfrak r:X_+\to X_-$. Then there is a building block $(Z_-,X_-)$ constructed by blowing up a base locus of an anti-canonical pencil in $W_-$ (see \autoref{thm weak Fano to ACyl building block}) such that the corresponding family of twisted connected sum $G_2$-manifolds $(Y_T,\phi_T)$ always contain closed rigid associative $3$-spheres $\widetilde P_T$ for all sufficiently large $T$.
\end{prop}

To prove \autoref{prop associative sphere} we need the following lemma;  see \cite[Lemma 2.5]{Menet2015} for a proof.

\begin{lemma}\label{lem base locus avoidence}Let $W$ be a semi-Fano $3$-fold and $X=X_\infty\in \abs{-K_W}$ be a smooth anti-canonical $K3$ surface such that ${-K_W}|_{_X}$ is very ample. Let $x\in X$ and $(y,v)\in \P(TX)$  be such that $x$ and $y$ are distinct. Then there is an anti-canonical $K3$ surface $X_0$ such that the base locus $B$ of the pencil $\abs{X_0:X_\infty}$ is smooth and 
 \begin{equation*}
 x\notin B,\ \  y\in B\ \text{and}\ \ T_{y}B=v.
 \end{equation*}
	\end{lemma}
	
\begin{proof}[{{{\normalfont{\textbf{Proof of \autoref{prop associative sphere}}}}}}]  Suppose $\ell_+\cap X_+=\{x_+\}$. Define $x_-:=\mathfrak r(x_+)\in X_-$. Since $N\ell_+\cong \mathcal O_{\P^1}(-1)\oplus \mathcal O_{\P^1}$ (cf. \autoref{rmk H1}), $\dim_\C H^0(\ell_+,N\ell_+)=1$ and therefore we can choose $v\in \mathbf P(T_{x_-}X_-)$ such that $v\cap \operatorname{im} \mathfrak r_*\operatorname{ev}_+=\{0\}$. By \autoref{lem base locus avoidence} we can choose an anti-canonical pencil in $W_-$ so that the base locus $B_-$ is smooth, $x_-\in B_-$ and $ T_{x_-}B_-=v$. Let $\pi:Z_-\to W_-$ be the blow up of $W_-$ along $B_-$. Consider the unobstructed line $\ell_-:=\pi^{-1}(x_-)\subset Z_-$. Then $\operatorname{im} \operatorname{ev}_-=v\in \mathbf P(T_{x_-}X_-)$. Since $H^0(\ell_\pm,N\ell_\pm(-x_\pm))\cong H^1(\ell_\pm,N\ell_\pm)=0$ (cf. \autoref{rmk H1}), by \autoref{thm gluing asso from ACyl holo} we complete the proof.
\end{proof}

\begin{example}\label{eg associative $3$-sphere}
\autoref{prop associative sphere} produces rigid associative $3$-spheres as follows.
\begin{enumerate}[label=(\roman*), leftmargin=*]	
\item We apply \autoref{prop associative sphere} to all very ample Fano $3$-folds $W_\pm$ that have been used in \cite{CHNP15} to obtain twisted connected sum $G_2$ manifolds. Let  $Z_+$ be a building block that comes from $W_+$. More explicitly, $\pi:Z_+\to W_+$ is a blow-up of $W_+$ along a base locus $B_+$ as described in  \autoref{thm weak Fano to ACyl building block}. Then for every $x\in B_+$ we can use the unobstructed holomorphic line $\pi^{-1}(x)$ in \autoref{prop associative sphere} to obtain a rigid associative $3$-sphere. 
\item \citet{Shokurov1979a} proved that every Fano $3$-fold $W$ of index $1$ except $\P^1\times \P^2$ contains a line $\ell$ with $-K_{W}\cdot\ell=1$. Any line in a general Fano $3$-fold of index $1$ and Picard rank $1$ is unobstructed \cite[Proposition 4.2.2 and Theorem 4.2.7]{Iskovskih1999}. In the Mori-Mukai list there are $97$ deformation types of Fano $3$-folds with very ample anti-canonical bundle and $8$ of them are of index $1$ and Picard rank $1$ namely,
$$\#_3^1, \#_4^1, \#_5^1,  \#_6^1,  \#_7^1,  \#_8^1, \#_9^1, \#_{10}^1$$
where $\#^\rho_n$ is the $n$-th Fano $3$-fold with Picard rank $\rho$ in the Mori-Mukai list (see \cite[Appendix 12.2]{Iskovskih1999}). The Fano $3$-folds with not very ample anti-canonical bundle \cite[Theorem 2.4.5, Theorem 2.1.16]{Iskovskih1999} are 
$$\#_1^1, \#_2^1, \#_{12}^1,  \#_1^2,  \#_2^2,  \#_3^2, \#_1^7, \#_1^8.$$
Therefore we are able to take $8$ and $97$ Fano $3$-folds as $W_+$ and $W_-$, respectively to apply \autoref{prop associative sphere} (as $\rk \Pic(W_+)+\rk \Pic(W_-)\leq 11$). In this way we obtain many rigid associative $3$-spheres. \qedhere
\end{enumerate}
\end{example}
\begin{remark}This list of examples is not exhaustive. The reader may find more examples by applying \autoref{prop associative sphere} to many other Fano $3$-folds or more generally to semi-Fano $3$-folds.
\end{remark}


\section{Associative submanifolds by gluing ACyl special Lagrangian $3$-folds}\label{sec Associative submanifolds by gluing of ACyl special Lagrangian $3$-folds} 
In this section, we rephrase \autoref{hyp gluing} in the setting where the ACyl associative submanifolds are given by ACyl special Lagrangian $3$-folds. In particular, using \autoref{thm gluing associative submanifolds}, we again construct another type of associative submanifolds in the twisted connected sum $G_2$-manifolds defined in \autoref{def TCS G2} (see \autoref{section TCS}). This is formalized in the following theorem.

\begin{theorem}\label{thm gluing with ACyl SL}Let $(V_\pm,\omega_\pm,\Omega_\pm)$ be a pair of ACyl Calabi--Yau $3$-folds with asymptotic cross sections $(X_\pm,\omega_{1}^\pm,\omega_{2}^\pm,\omega_{3}^\pm)$ having a hyperkähler rotation $\mathfrak r:X_+\to X_-$. Let  $\{(Y_T,\phi_T)\}$ be the family of  $G_2$-manifolds obtained from the twisted connected sum construction. Let $L_\pm$ be a matching pair of connected ACyl embedded special Lagrangian $3$-folds with connected cross sections $\Sigma_{s_\pm}:=\{e^{is_\pm}\}\times\Sigma_\pm$ in $V_\pm$, where $\Sigma_\pm$ are $I_3^\pm$-holomorphic curves in $X_\pm$. Let $i_\pm^*:H^1(L_\pm,\R)\to H^1(\Sigma_\pm,\R)$ be the map induced by the inclusion $\Upsilon_\pm\circ \iota_\pm:\Sigma_\pm\cong\{T\}\times \Sigma_\pm\to L_\pm$ for any large $T$. Assume that 
	\begin{itemize} 
	\item $\mathfrak r(\Sigma_+)=\Sigma_-$,
	\item $b_2(L_\pm)=0$,
	\item $\operatorname{im}(I_3^+\circ i_+^*) \cap \operatorname{im}{(\mathfrak r^*\circ i_-^*)}=\{0\}$.
	\end{itemize}
Then $L_{s_\mp}:=\{e^{is_\mp}\}\times L_\pm$ satisfies \autoref{hyp gluing} and hence we obtain a family of closed rigid associative submanifolds $\widetilde P_T$ in $(Y_T,\phi_T)$ for all sufficiently large $T$ which are diffeomorphic to the twisted connected sum $L_+ \#_\mathfrak r L_-$.
\end{theorem}

The proof of the above theorem relies on the following observations. Consider \autoref{eg ACyl asso from ACyl SL}(i); with a slight abuse of notation we reuse the symbols $L$ and $\Sigma$ from that example. 

\begin{definition}\label{def PhiL}
Define  isometries $$\Phi_{L}:C^\infty(NL)\to \Omega^0(L,\R)\oplus \Omega^1(L,\R),$$ $$\Phi_\Sigma:\Omega^0(\Sigma,\R)\oplus \Omega^0(\Sigma,\R)\oplus C^\infty(N_X\Sigma) \to \Omega^0(\Sigma,\R)\oplus \Omega^0(\Sigma,\R)\oplus \Omega^1(\Sigma,\R)$$  by
$$\Phi_{L}(u):=(\inp{\partial_\theta}{u},(\partial_\theta\times u)^\flat) \qandq  \Phi_\Sigma(f_1,f_2,u):=(f_1,-f_2,\iota_u\omega_1).$$
Here $\theta$ denotes the coordinate of $S^1$ factor of the $G_2$-manifold obtained by taking its product with an ACyl Calabi--Yau $3$-fold. 

Let $\bD_L$ and $\bD_\Sigma$ be the Fueter operators defined in \autoref{eq Fueter operator} and \autoref{eq: Fueter on Riemann surface}, respectively. Define the following operators 
 \[\check{\mathbf D}_L:=\Phi_{L}{\mathbf D}_{L}\Phi_{L}^{-1} \qandq \check{\mathbf D}_\Sigma:=\Phi_{\Sigma}{\mathbf D}_{\Sigma}\Phi_{\Sigma}^{-1}.\qedhere\]
 \end{definition}
\begin{lemma}\label{Lem Fueter Sl}The following holds.
\[\check{\mathbf D}_L=
\begin{bmatrix}
0 & d^* \\
d & *d \\
\end{bmatrix}
\qandq 
\check{\mathbf D}_\Sigma=
\begin{bmatrix}
0 & 0& d_\Sigma^{*} \\
0 &0 & -*d_\Sigma\\
d_\Sigma & *d_\Sigma &0
\end{bmatrix}.\]
\end{lemma}
\begin{proof} A direct computation shows that $\Phi_{L}^{-1}(f,\sigma)=f\partial_\theta-\partial_\theta\times\sigma_\sharp$, where and $\sigma_\sharp$ is the metric dual to the $1$-form $\sigma$. Now given a $1$-form $\sigma\in \Omega^1(L,\R)$, we observe that $\nabla^\perp_L(\partial_\theta\times\sigma_\sharp)=\partial_\theta\times \nabla^\parallel_L\sigma_\sharp$. Therefore, $${\Phi_{L}}_*{\nabla}^\perp_{L}(f,\sigma)=\Phi_{L}{\nabla}^\perp_{L}\Phi_{L}^{-1}(f,\sigma)=(\nabla f,{\nabla}^\parallel_{L}\sigma).$$ Denote the Clifford multiplication for $\mathbf D_L$ by $\gamma$. We see that $${\Phi_{L}} \gamma(v) \Phi_{L}^{-1}(f,\sigma)=(-\inp{v}{\sigma_\sharp},fv^\flat+ (v\times \sigma_\sharp )^\flat)=(-\iota_v\sigma,fv^\flat+*(v^\flat\wedge \sigma) ).$$  
	Hence, $\check{\mathbf D}_L={\Phi_{L}}_*\gamma({\Phi_{L}}_*{\nabla}^\perp_{L})$ has to be as in the statement of the proposition.
	
	Replacing $L$ by the associative cylinder $C=\R\times \Sigma$, we see that $NC=\pi^*(N_Z\Sigma)$,  $N_Z\Sigma=\R^2\oplus N_X\Sigma$ and $$\Phi_C:\Omega^0(C,\R)\oplus \Omega^0(C,\R)\oplus C^\infty(\R, C^\infty(N_X\Sigma) )\to \Omega^0(C,\R)\oplus \Omega^0(C,\R)\oplus C^\infty(\R,\Omega^1(\Sigma,\R))$$
is given by $\Phi_C(f_1,f_2,u)=(f_1,-f_2,\iota_u\omega_1)$. Since ${\mathbf D}_C=J\partial_t+{\mathbf D}_\Sigma$, it completes the proof.
\end{proof}

\begin{lemma}\label{Lem ACyl asso ACyl SL }Assume the asymptotic cross section $\Sigma$ of $L$ is connected. Then the de Rham cohomology class map $[\cdot]$ induces the isomorphisms:
\[[\cdot]: \ker{\check \bD}_{L,0}\to  H^0(L,\R)\oplus H^1(L,\R)\]
and 
 \[[\cdot]:\ker \check{\mathbf D}_{\Sigma} \to H^0(\Sigma,\R)\oplus H^0(\Sigma,\R)\oplus H^1(\Sigma,\R). \]
Moreover, the following diagram commutes:
\[\begin{tikzcd}
\ker{\mathbf D}_{L,0} \arrow{r}{\Phi_{L}} \arrow[swap]{d}{\iota_{\infty}} & \ker{\check \bD}_{L,0}\cong H^0(L,\R)\oplus H^1(L,\R) \arrow{d}{i^*\oplus(0\oplus i^*)} \\
\ker{\mathbf D}_{\Sigma} \arrow{r}{\Phi_{\Sigma}} & \ker{\check \bD}_{\Sigma}\cong H^0(\Sigma,\R)\oplus \big(H^0(\Sigma,\R)\oplus H^1(\Sigma,\R)\big). \end{tikzcd}
\]  
\end{lemma}
\begin{proof}
Set $\Omega^k_\lambda(L,\R):=\{\sigma\in \Omega^k(L,\R):\abs{\nabla^l \sigma}=O(e^{\lambda t})\ \text{as}\ t\to \infty, \forall l\in \N\cup\{0\}\}.$ 
We claim that the linear map $$[\cdot]_1:\mathcal H^1_0:=\{\sigma\in \Omega_0^1(L,\R):d\sigma=0, d^*\sigma=0\}\to H^1(L,\R)$$
given by $\sigma\mapsto [\sigma]$, is an isomorphism. The proof of this claim can be found in the literature; for e.g.\cite[Section 5.2]{Nordstrom2008}, but for reader's convenience we include a proof here.

Suppose $[\sigma]=0$ for some $\sigma\in \mathcal H^1_0 $. Then $\sigma=dh$ for some harmonic function $h=O(\log t)$ on $L$; in fact $h\in \Omega^0_\lambda(L,\R)$ for any $\lambda>0$. Consider the Laplace operator:
\[ \Delta_\lambda:=\Delta_L:  \Omega^0_\lambda(L,\R)\to  \Omega^0_\lambda(L,\R).\]
Fix $\lambda>0$ sufficiently small.  An argument with integration by parts proves that $\ker \Delta_{-\lambda}=0$. Since $\coker \Delta_{\pm \lambda}\cong \ker \Delta_{\mp\lambda}$ and  $\ind \Delta_\lambda-\ind \Delta_{-\lambda}=2b^0(\Sigma)$ \cite[Theorem 4.15]{Marshall2002}, it follows that $\dim \ker \Delta_\lambda = b^0(\Sigma)$. As $\Sigma$ is connected and the constant functions already lie in $\ker \Delta_\lambda$, the function $h$ must be constant, and hence $\sigma = 0$. This shows that $[\cdot]_1$ is injective. 

 Given $[\eta]\in H^1(L,\R)$,  we choose the harmonic representative, say $\sigma_\Sigma$, of the image of $[\eta]$ under the restriction map $i^*:H^1(L,\R)\to H^1(\Sigma,\R)$. Over the end, under the canonical identifications, $\eta$ can be expressed as:
 \[\eta=\sigma_\Sigma+df, \quad \text{for some} \quad f=f(T_0)+\int_{T_0}^t\iota_{\partial_t}\eta.\]
Then $\widetilde \eta:=\eta-d(\chi_{T_0}f)$ is exactly $\sigma_\Sigma$ over the end. Moreover, integration by parts implies that $\int_L d^*\widetilde \eta=0$. Since $\coker \Delta_{-\lambda} \cong \ker \Delta_{\lambda} $ is one dimensional as above, there exists $\widetilde h \in  \Omega^0_{-\lambda}(L,\R)$ with  $\lambda>0$ sufficiently small such that $d^*\widetilde\eta=\Delta_{L}\widetilde h$. Define $\sigma:=\widetilde \eta-d\widetilde h\in \mathcal H^1_0$. Then $[\eta]=[\sigma]$ and hence $[\cdot]_1$ is surjective. This proves the above claim.

 The argument above also implies that every $\sigma \in \mathcal{H}^1_0$ is asymptotic to the harmonic representative $\sigma_\Sigma$ of the image of $[\sigma]$ under the above restriction map $i^*$. The above explains that any bounded harmonic function on $L$ is a constant and therefore: $(f,\sigma)\in\ker{\check{\mathbf D}}_{L,0}$ if and only if 
$df=0$, $d\sigma=0$, $d^*\sigma=0$. This completes the proof.
\end{proof}

\begin{proof}[{{{\normalfont{\textbf{Proof of \autoref{thm gluing with ACyl SL}}}}}}]
 Since $\mathfrak r(\Sigma_+)=\Sigma_-$, therefore by definition of $f$ we have $f(\Sigma_{s_+})=\Sigma_{s_-}$. By \autoref{Lem ACyl asso ACyl SL } we see that $\iota^\pm_\infty$ is injective if and only if $i_\pm^*:H^1(L_\pm,\R)\to H^1(\Sigma_\pm,\R)$ is injective. Since $L_\pm$ and $\Sigma_\pm$ are connected, this is again equivalent to $H^1_{\operatorname{cs}}(L_\pm,\R)\cong H_2(L_\pm,\R)=0$.  By \autoref{Lem ACyl asso ACyl SL }, over $H^1(\Sigma_\pm,\R)$: 
	$$\mathfrak r^*\Phi_{\Sigma_-}\mathfrak r_*\Phi_{\Sigma_+}^{-1}=I_3^+.$$
	Therefore, $\operatorname{im}(f_*\iota^+_\infty) \cap \operatorname{im}{\iota^-_\infty}=\{0\}$ if and only if $\operatorname{im}(I_3^+\circ i_+^*) \cap \operatorname{im}{(\mathfrak r^*\circ i_-^*)}=\{0\}$.
	\end{proof}
The last two conditions in the assumption of \autoref{thm gluing with ACyl SL} are automatically satisfied if $b^1(L_\pm)=0$. A simple way to construct ACyl special Lagrangian $3$-folds is to look for anti-holomorphic involutions on building blocks; see \autoref{eg ACyl asso from ACyl SL}. Given the following hypothesis on a building block that admits an anti-holomorphic involution, we construct another building block, and using \autoref{thm gluing with ACyl SL} we produce a closed associative submanifold in the associated twisted connected sum $G_2$-manifold. In particular, we construct associative submanifolds diffeomorphic to $\R\P^3$ or $\R\P^3\# \R\P^3$.
	   
\begin{hypothesis}\label{hyp assso from involution SL} The building block $(Z,X)$ satisfies the following:
	\begin{itemize} 
\item $Z$ admits an anti-holomorphic involution $\sigma:Z\to Z$ preserving $X$,	
\item $X$ admits a non-symplectic involution $\rho$ that commutes with ${\sigma}|_{X}$, 
\item  $\operatorname{Fix}_{\sigma}(X)$ has a connected component $\Sigma$ with $\rho(\Sigma)=\Sigma$, $b^1({\Sigma}/{\<\rho\>})=0$ and  $\Sigma\cap \operatorname{Fix}_{\rho}(X)=\emptyset$.
\item $\operatorname{Fix}_{\sigma}(Z)\setminus X$ has a connected, non-compact component $L$ with $\partial L=\Sigma$ and $b^1(L)=0$.
\end{itemize}
\end{hypothesis}

\begin{prop}\label{prop asso from ACyl SL by involution}Let $(Z_+,X_+)$ be a building block satisfying \autoref{hyp assso from involution SL}.  
Then $V_{+}:=Z_{+}\setminus X_{+}$ admits an ACyl Calabi--Yau structure $(\omega_{+},\Omega_{+})$ such that $\sigma_{+}$ is an anti-holomorphic involutive isometry on $V_{+}$. If the hyperkähler structure of $X_+$ is $(\omega_{1}^+,\omega_{2}^+,\omega_{3}^+)$, then take $X_-$ to be the same manifold $X_+$ but with hyperkähler structure $(\omega_{2}^+,\omega_{1}^+,-\omega_{3}^+)$ so that the identity map $\id:X_+\to X_-$ becomes a hyperkähler rotation. Then $\rho_-:=\rho_+\circ\sigma_+$ is a non-symplectic involution on $X_-$.
 Define $Z_-$ to be the blow-up of $W_-:=\tfrac{\C\P^1\times X_-}{\<\iota\times{\rho_-}\>}$ along the fixed point locus, where $\iota:\C\P^1\to \C\P^1$ is  $\iota(z)=1/z$ as in \autoref{thm Kovalev-Lee} and $\widetilde L_-$ to be the proper transform of $\tfrac{\R\P^1\times \Sigma_+}{\<\iota\times{\rho_+}\>}$. Denote $L_-:=\widetilde L_-\setminus X_-$. Then the pair of ACyl special Lagrangians $L_\pm$ satisfies the conditions in \autoref{thm gluing with ACyl SL} and hence we obtain a family of closed rigid associative submanifolds $\widetilde P_T$ in $(Y_T,\phi_T)$ for all sufficiently large $T$.
\end{prop}

\begin{proof}
By \cite[Proposition 5.2]{Kovalev2013}, $V_{+}:=Z_{+}\setminus X_{+}$ admits an ACyl Calabi--Yau structure $(\omega_{+},\Omega_{+})$ such that $\sigma_{+}$ is an anti-holomorphic involutive isometry on $V_{+}$. Now by definition we have $\rho_+^*(\omega_{1}^+)=\omega_{1}^+$ and $\rho_+^*(\omega_{2}^++i\omega_{3}^+)=-\omega_{2}^+-i\omega_{3}^+$. Also $\sigma_+^*(\omega_{1}^+)=-\omega_{1}^+$ and $\sigma_+^*(\omega_{2}^++i\omega_{3}^+)=-\omega_{2}^++i\omega_{3}^+$. We can now easily check that $\rho_-$ is a non-symplectic involution on $X_-$. Since ${\Sigma_+}/{\<\rho_+\>}$ is a deformation retract of $L_-$, therefore $b^1(L_-)=0$. Thus the pair of ACyl special Lagrangians $L_\pm$ satisfies the conditions of \autoref{thm gluing with ACyl SL}. 
 \end{proof}

\begin{example}[{\citet{Nordstrom2013}}]\label{eg Nordstrom RP3} Let $W_+$ be the Fano $3$-fold in $\C\P^4$ defined by the quartic polynomial 
$$Q(z_0,z_1,z_2,z_3,z_4)=-z_0^4+z_1^4+z_2^4+z_3^4+z_4^4=0.$$
Then $X_\infty:=\{z_4=0\}$ and $X_0:=\{z_0=0\}$ are anti-canonical divisors in $W_+$. The base locus $B$ of the anti-canonical pencil $\abs{X_0,X_\infty}$ is $\{z_0=0,z_4=0\}$. Let $Z_+$ be the blow-up of $W_+$ along $B$ and let $X_+$ be the proper transform of $X_\infty$ as described in  \autoref{thm weak Fano to ACyl building block}. Then $(Z_+,X_+)$ is a building block. As the complex conjugation in $\C\P^4$ acts on $X_\infty$ and $B$, it induces an anti-holomorphic involution ${\sigma_+}$ on $Z_+$ which acts also on $X_+$ by \cite[pg. 19]{Kovalev2013}. The involution $\rho_\infty:X_\infty\to X_\infty$ defined by
$$\rho_\infty(z_0,z_1,z_2,z_3,0)=(-z_0,z_1,z_2,z_3,0)$$
 induces a non-symplectic involution $\rho_+$ on $X_+$ which commutes with ${\sigma_{+}}|_{X_+}$ and acts freely on $\Sigma_+:=\operatorname{Fix}_{\sigma_+}(X_+)$. Since $\operatorname{Fix}_{\sigma_+}(W_+)$ is disjoint from $B$ therefore $\operatorname{Fix}_{\sigma_+}(Z_+)$ is homeomorphic to $W_+\cap \R\P^4\cong S^3$ and $\Sigma_+:=\operatorname{Fix}_{\sigma_+}(X_+)$ is homeomorphic to $X_\infty\cap \R\P^4\cong S^2$. The involution $\rho_\infty$ acts on $S^2$ as an antipodal map. Thus $\operatorname{Fix}_{\sigma_+}(Z_+)\setminus\Sigma_+$ is a disjoint union of two $3$-balls. Let $L_+$ be any one of these $3$-balls. Applying \autoref{prop asso from ACyl SL by involution} we obtain a family of closed rigid associative submanifolds $\widetilde P_T$ in  $(Y_T,\phi_T)$, each of which is diffeomorphic to $\R\P^3$. 
\end{example}
	\begin{example}\label{eg RP3 RP3 first} Let $X_\infty$ be the $K3$ surface with non-symplectic involution $\rho_\infty$ as described in  \autoref{eg Nordstrom RP3}.
Let $W_+:=\tfrac{\C\P^1\times X_\infty}{\<\iota\times{\rho_\infty}\>}$ and let $Z_+$ be the blow-up of $W_+$ as described in  \autoref{thm Kovalev-Lee}. Let $X_+$ be the proper transform of $\{\infty\}\times X_\infty$ as described in  \autoref{thm Kovalev-Lee}. Then $(Z_+,X_+)$ is a building block. As the complex conjugation in $\C\P^3$ acts on $X_\infty$ and $\operatorname{Fix}_{\rho_\infty}(X_\infty)$, it induces an anti-holomorphic involution ${\sigma_+}$ on $Z_+$ which acts also on $X_+$ by \cite[pg. 19]{Kovalev2013}. The non-symplectic involution $\rho_\infty$ on $X_+$ commutes with ${\sigma_{+}}|_{X_+}$ and again acts on $\Sigma_+:=\operatorname{Fix}_{\sigma_+}(X_+)\cong S^2$ as an antipodal map. Let $\widetilde L_+$ be the proper transform of $\frac{\R\P^1\times \Sigma_\infty}{\<\iota\times{\rho_\infty}\>}$, where $\Sigma_\infty:=\operatorname{Fix}_{\rho_\infty}(X_\infty)$. Applying \autoref{prop asso from ACyl SL by involution} we obtain a family of closed rigid associative submanifolds $\widetilde P_T$  in $(Y_T,\phi_T)$, each of which is diffeomorphic to $\R\P^3\#\R\P^3$.
\end{example}

\begin{example}\label{eg RP3 RP3 2nd} Let $X_\infty\subset \C\P(1,1,1,3)$ be the $K3$ surface defined by the polynomial 
$$P(z_0,z_1,z_2,z_3)=z_3^2-z_0^6-z_1^6-z_2^6=0$$ 
which is a double cover of $\C\P^2$ branched along the curve $\{z_0^6+z_1^6+z_2^6=0\}$. Define a non-symplectic involution $\rho_\infty$ on $X_\infty$ taking $(z_0,z_1,z_2,z_3)\mapsto (z_0,z_1,z_2,-z_3)$. Let $W_+:=\tfrac{\C\P^1\times X_\infty}{\<\iota\times{\rho_\infty}\>}$ and let $Z_+$ be the blow-up of $W_+$ as described in  \autoref{thm Kovalev-Lee}. Let $X_+$ be the proper transform of $\{\infty\}\times X_\infty$. Then $(Z_+,X_+)$ is a building block. The complex conjugation in $\C\P(1,1,1,3)$ induces an anti-holomorphic involution ${\sigma_+}$ on $Z_+$ which acts also on $X_+$ by \cite[pg. 19]{Kovalev2013}. The  non-symplectic involution $\rho_\infty$ on $X_+$ commutes with ${\sigma_{+}}|_{X_+}$ and acts freely on $\Sigma_+:=\operatorname{Fix}_{\sigma_+}(X_+)$. Let $\widetilde L_+$ be the proper transform of $\tfrac{\R\P^1\times \Sigma_\infty}{\<\iota\times{\rho_\infty}\>}$, where $\Sigma_\infty:=\operatorname{Fix}_{\rho_\infty}(X_\infty)$. Under the homeomorphism $\R\P(1,1,1,3)\to \R\P^4$ taking $(x_0,x_1,x_2,x_3)\to (x_0^3,x_1^3,x_2^3,x_3)$ we see that ${\Sigma_+}\cong S^2$ and the involution $\rho_\infty$ acts on $S^2$ as an antipodal map. Applying \autoref{prop asso from ACyl SL by involution} we obtain a family of closed rigid associative submanifolds $\widetilde P_T$ in $(Y_T,\phi_T)$, each of which is diffeomorphic to $\R\P^3\#\R\P^3$.
\end{example}

\begin{remark}This list of examples is not exhaustive. The reader may find more examples by applying \autoref{prop asso from ACyl SL by involution} to $K3$ surfaces having commuting non-symplectic involutions and anti-holomorphic involutions, studied by \citet{Nikulin2005, Nikulin2007}, and  by \citet{ReidegeldCAG2023, ReidegeldDGA2023}.
\end{remark}

\printreferences

\noindent
\author{Simons Center for Geometry and Physics, State University of New York, Stony Brook, NY 11794} \\ E-mail address: \href{ mailto:gbera@scgp.stonybrook.edu}{gbera@scgp.stonybrook.edu} 

\end{document}

%% file: preamble/preamble.tex
\usepackage[utf8]{inputenc}
\usepackage[T1]{fontenc}
\usepackage{microtype}

\usepackage{calc}
\usepackage[a4paper]{geometry}

\usepackage{amsmath}
\usepackage{amsthm}
\usepackage{tikz-cd}
\usetikzlibrary{arrows} 
\tikzset{
  commutative diagrams/.cd, 
  arrow style=tikz, 
  diagrams={>=stealth}
}
\usetikzlibrary{matrix,decorations.pathreplacing,calc}

\usepackage{textcomp}
\usepackage[sb]{libertine}
\usepackage[varqu,varl]{zi4}%
\usepackage[libertine,bigdelims,vvarbb]{newtxmath}

\usepackage[scr=boondox,cal=euler]{mathalfa}

\usepackage{marvosym}

\usepackage{slashed}
\usepackage{esint} 
\usepackage[english]{babel}

\usepackage{imakeidx}
\indexsetup{noclearpage}
\makeindex[intoc]

\usepackage{csquotes}
\usepackage[
  backend=biber,
  hyperref=true,
  backref=true,
  isbn=false,
  doi=true,
  natbib=true,
  eprint=true,
  useprefix=true,
  maxcitenames=99,
  maxbibnames=99,  
  maxalphanames=99, 
  minalphanames=99,
  safeinputenc,
  style=alphabetic,
  citestyle=alphabetic,
  block=space,
  datamodel=preamble/ext-eprint,
  sorting=nyt
]{biblatex}
\definecolor{darkblue}{HTML}{0000A6}

\usepackage[
  bookmarksnumbered = true,
  hypertexnames     = false,
  colorlinks        = true,
  citecolor         = darkblue,
  linkcolor         = darkblue,
  urlcolor          = darkblue,
  breaklinks
]{hyperref}

\DeclareFieldFormat{url}{%
  \href{#1}{\ComputerMouse}
}
\DeclareFieldFormat{doi}{%
  \mkbibacro{DOI}\addcolon\space\href{https://doi.org/#1}{#1}
}
\makeatletter
\DeclareFieldFormat{arxiv}{%
  arXiv\addcolon\space\href{http://arxiv.org/\abx@arxivpath/#1}{#1}
}
\makeatother
\DeclareFieldFormat{mr}{%
  MR\addcolon\space\href{http://www.ams.org/mathscinet-getitem?mr=MR#1}{#1}
}
\DeclareFieldFormat{zbl}{%
  Zbl\addcolon\space\href{http://zbmath.org/?q=an:#1}{#1}
}
\renewbibmacro*{eprint}{%
  \printfield{arxiv}%
  \newunit\newblock
  \printfield{mr}%
  \newunit\newblock
  \printfield{zbl}%
  \newunit\newblock
  \iffieldundef{eprinttype}
  {\printfield{eprint}}
  {\printfield[eprint:\strfield{eprinttype}]{eprint}}
}

\AtEveryBibitem{%
  \clearlist{address}%
}
\DeclareFieldFormat[article,inproceedings,inbook,incollection,thesis]{title}{\textit{#1}}
\renewbibmacro{in:}{}
\newcommand{\printreferences}{\raggedright\printbibliography[heading=bibintoc]}
\usepackage[inline,shortlabels]{enumitem}

\usepackage{subcaption}

\usepackage{etoolbox}
\ifundef{\abstract}{}{\patchcmd{\abstract}%
    {\quotation}{\quotation\noindent\ignorespaces}{}{}}

\usepackage[super]{nth}

\input{preamble/envs.tex}
\input{preamble/macros.tex}
\input{preamble/letters.tex}


%% file: preamble/envs.tex
\usepackage{thmtools}

\numberwithin{equation}{section}
\renewcommand{\qedsymbol}{$\blacksquare$}

\newcommand{\CorollaryQED}{\qedsymbol}

\newcommand{\ConjectureQED}{$\square$}
\newcommand{\SituationQED}{$\times$}
\newcommand{\DefinitionQED}{$\spadesuit$}
\newcommand{\NotationQED}{$\blacktriangleright$}
\newcommand{\ExampleQED}{$\bullet$}
\newcommand{\RemarkQED}{$\clubsuit$}

\declaretheorem[numberlike=equation,]{theorem}
\declaretheorem[numbered=no,name=Theorem]{theorem*}
\declaretheorem[numberlike=equation,name=Lemma]{lemma}
\declaretheorem[numberlike=equation,name=Proposition]{prop}

\declaretheorem[numberlike=equation,name=Hypothesis]{hypothesis}

\declaretheorem[numberlike=equation,name=Definition,style=definition,qed=\DefinitionQED]{definition}
\declaretheorem[numbered=no,name=Definition,style=definition,qed=\DefinitionQED]{definition*}
\declaretheorem[numberlike=equation,name=Notation,style=definition,qed=\NotationQED]{notation}

\declaretheorem[numberlike=equation,style=definition,qed=\ExampleQED]{example}

\declaretheorem[numberlike=equation,style=remark,qed=\RemarkQED]{remark}
\declaretheorem[numbered=no,style=remark,name=Remark,qed=\RemarkQED]{remark*}

\def\makeautorefname#1#2{\AtBeginDocument{\expandafter\def\csname#1autorefname\endcsname{#2}}}
\makeautorefname{table}{Table}        
\makeautorefname{chapter}{Chapter}
\makeautorefname{section}{Section}
\makeautorefname{subsection}{Section}
\makeautorefname{subsubsection}{Section}
\makeautorefname{footnote}{Footnote}
\AtBeginDocument{\def\itemautorefname~#1\null{(#1)\null}}
\AtBeginDocument{\def\equationautorefname~#1\null{(#1)\null}}

\numberwithin{substep}{step}
\makeautorefname{step}{Step}
\makeautorefname{substep}{Step}

\makeautorefname{case}{Case}
\makeautorefname{substep}{Step}

\setlist[description]{leftmargin=!,labelindent=1em}
\setlist[enumerate]{label={\rm (\arabic*)},ref=\arabic*}
\setlist[enumerate,2]{label={\rm (\alph*)},ref=\theenumi.\alph*}
\setlist[enumerate,3]{label={\rm (\roman*)},ref=\theenumii.\roman*}

%% file: preamble/macros.tex
\let\C\undefined

\usepackage{bm}
\usepackage{mathtools} 
\usepackage{stmaryrd} 

\DeclareFontFamily{U}{mathx}{\hyphenchar\font45}
\DeclareFontShape{U}{mathx}{m}{n}{
      <5> <6> <7> <8> <9> <10>
      <10.95> <12> <14.4> <17.28> <20.74> <24.88>
      mathx10
      }{}
\DeclareSymbolFont{mathx}{U}{mathx}{m}{n}
\DeclareFontSubstitution{U}{mathx}{m}{n}
\DeclareMathAccent{\widecheck}{0}{mathx}{"71}
\DeclareMathAccent{\wideparen}{0}{mathx}{"75}

\DeclareMathOperator{\graph}{graph}
\DeclareMathOperator{\HF}{\HF}

\DeclareMathOperator{\Pic}{Pic}

\DeclareMathOperator{\coker}{coker}

\DeclareMathOperator{\im}{im}
\DeclareMathOperator{\ind}{index}

\DeclareMathOperator{\rk}{rk}

\DeclareMathOperator{\spec}{spec}

\DeclarePairedDelimiter{\norm}{\|}{\|}

\DeclarePairedDelimiterX{\inp}[2]{\langle}{\rangle}{#1, #2}

\DeclarePairedDelimiter{\abs}{\lvert}{\rvert}

\def\({\left(}
\def\){\right)}
\def\<{\left\langle}
\def\>{\right\rangle}

\newcommand{\C}{{\mathbf{C}}}

\newcommand{\N}{{\mathbf{N}}}

\newcommand{\R}{\mathbf{R}}

\newcommand{\SU}{\mathrm{SU}}

\newcommand{\Z}{\mathbf{Z}}

\newcommand{\id}{\mathbf 1}

\newcommand{\qandq}{\quad\text{and}\quad}

\renewcommand{\epsilon}{\varepsilon}

\newcommand{\vol}{\mathrm{vol}}

\renewcommand{\P}{\mathbf{P}}
\renewcommand{\Re}{\operatorname{Re}}

\renewcommand{\emptyset}{\varnothing}
\renewcommand{\setminus}{{\backslash}}

\renewcommand{\leq}{\leqslant}
\renewcommand{\geq}{\geqslant}

\newcommand{\tn}{\otimes}

\makeatletter
\renewcommand*\env@matrix[1][*\c@MaxMatrixCols c]{%
  \hskip -\arraycolsep
  \let\@ifnextchar\new@ifnextchar
  \array{#1}}

\renewcommand\xleftrightarrow[2][]{%
  \ext@arrow 9999{\longleftrightarrowfill@}{#1}{#2}}
\newcommand\longleftrightarrowfill@{%
  \arrowfill@\leftarrow\relbar\rightarrow}
\makeatother


%% file: preamble/letters.tex
\newcommand{\bD}{{\mathbf{D}}}

\newcommand{\cL}{\mathcal{L}}

\newcommand{\sC}{\mathscr{C}}